\documentclass[reqno]{amsart}
\usepackage[foot]{amsaddr}

\usepackage[utf8]{inputenc}

\usepackage{amsthm, amsmath, amssymb}
\usepackage{mathtools}
\usepackage{hyperref}
\usepackage{fullpage}

\usepackage{graphicx}
\usepackage[table]{xcolor}
\usepackage{subcaption}
\usepackage{caption}
\usepackage{float}

\usepackage{tikz-cd}
\usepackage{tikz}
\usetikzlibrary{arrows}
\usetikzlibrary{arrows.meta}
\usetikzlibrary{backgrounds}
\tikzset{    
    mypoint/.style={
        circle,
        draw,
        inner sep=.3mm
        },  
    whitepoint/.style={
        fill=white, 
        mypoint
        },  
    blackpoint/.style={
        fill=black, 
        mypoint
        },  
    textnode/.style={
        text height=2.5ex, 
        text depth=1ex
        },  
    }
\usepackage{pgfplots}
\DeclareUnicodeCharacter{2212}{−}
\usepgfplotslibrary{groupplots,dateplot}
\usetikzlibrary{patterns,shapes.arrows}
\pgfplotsset{compat=newest}

\newcommand{\R}{\mathbb{R}}
\newcommand{\N}{\mathbb{N}}

\newcommand{\x}{\mathbf{x}}
\newcommand{\X}{\mathbf{X}}
\newcommand{\y}{\mathbf{y}}
\newcommand{\Y}{\mathbf{Y}}

\newcommand{\Dgm}{\mathrm{Dgm}}

\newcommand{\z}{\mathbf{z}}

\newcommand{\basis}{\mathbf{b}}

\newcommand{\CDkernel}[2]{{K}^{#1}_{#2}}
\newcommand{\CDpol}[2]{{P}^{#1}_{#2}}

\newcommand{\samples}{\mathcal{X}}
\newcommand{\sampless}{\mathcal{Y}}

\newcommand{\dist}{\mathrm{dist}}

\newcommand{\persmod}[2]{\mathbb{CD}(#2, #1)}

\newtheorem{theorem}{Theorem}
\newtheorem{corollary}[theorem]{Corollary}
\newtheorem{lemma}[theorem]{Lemma}
\newtheorem{definition}[theorem]{Definition}
\newtheorem{proposition}[theorem]{Proposition}
\newtheorem{assumption}[theorem]{Assumption}

\theoremstyle{definition}
\newtheorem{example}[theorem]{Example}

\title{The Christoffel-Darboux kernel for topological data analysis}

\author{Pepijn Roos Hoefgeest\textsuperscript{*} \and Lucas Slot\textsuperscript{$\dagger$}}
\address{\textsuperscript{*}Vrije Universiteit (VU),  Amsterdam.}
\email{p.e.r.rooshoefgeest@vu.nl}

\address{\textsuperscript{$\dagger$}ETH Z\"urich \and Centrum Wiskunde \& Informatica (CWI), Amsterdam}
\email{lucas.slot@inf.ethz.ch}
\date{\today}

\begin{document}

\begin{abstract}
    Persistent homology has been widely used to study the topology of point clouds in $\R^n$. Standard approaches are very sensitive to outliers, and their computational complexity depends badly on the number of data points. 
    In this paper we introduce a novel persistence module for a point cloud using the theory of Christoffel-Darboux kernels. This module is robust to (statistical) outliers in the data, and can be computed in time linear in the number of data points. We illustrate the benefits and limitations of our new module with various numerical examples in $\R^n$, for $n=1, 2, 3$. Our work expands upon recent applications of
    Christoffel-Darboux kernels in the context of statistical data analysis and geometric inference~\cite{LPP2022}. There, these kernels are used to construct a polynomial whose level sets capture the geometry of a point cloud in a precise sense. We show that the persistent homology associated to the sublevel set filtration of this polynomial is stable with respect to the Wasserstein distance. Moreover, we show that the persistent homology of this filtration can be computed in singly exponential time in the ambient dimension $n$, using a recent algorithm of Basu~\&~Karisani~\cite{Basu2022}.
\end{abstract}

\maketitle

\section{Introduction}

Persistent homology is a central tool in the field of topological data analysis. It was developed in the early 2000s in order to extract topological and geometric information out of point-cloud data. Since discrete points in $\R^n$ do not have any meaningful topological features in and of themselves, one needs to find a way to construct an `interesting' topological space out of them. An obvious approach is to look at the collection of balls of radius $r$ centered around the data points. When the radius $r$ is chosen correctly, these balls will intersect in ways that reflect the topology of the set the data is sampled from. However, it is not clear a priori which radius should be chosen, and in fact, a \emph{single} `correct' choice need not even exist. The solution is to look at \emph{all} radii $r \geq 0$, and to track which topological features \emph{persist} over time as $r$ increases. More concretely: if $r \leq r'$, then the balls of radius $r$ include into those of radius $r'$, and this collection of inclusions forms what is called a \emph{filtration}. Persistent homology tracks `birth-death events' of homology classes in such a filtration, see Figure~\ref{FIG:persdia}.
Classical approaches to obtain a filtration of topological spaces out of a point cloud, such as the \emph{\v{C}ech filtration} (outlined above) and the \emph{Vietoris-Rips filtration}~\cite{zomorodian} suffer from two main problems:
\begin{enumerate}
    \item The complexity of  computing the persistent homology depends badly on the number of data points.
    \item The persistent homology of these filtrations is very sensitive to outliers in the data.
\end{enumerate}
These issues have been addressed in the literature in several ways.
\emph{Alpha complexes}~\cite{alphashapes} are used to compute the persistent homology of the \v{C}ech filtration efficiently by first intersecting the metric balls with a Voronoi diagram. 
\emph{Witness complexes}~\cite{WitnessComplex} build a small simplicial complex based on a subsample of the data, thus reducing computational complexity. Heuristically, these subsamples may be chosen to reduce sensitivity to outliers in the full data set, although this effect remains hard to quantify~\cite{RobustWitness}.
Chazal et al.~\cite{Chazal2011} introduce the \emph{distance-to-measure} function, which they apply to perform geometric inference of point clouds in $\R^n$. The key feature of this function is that it is stable with respect to the \emph{Wasserstein distance}, implying robustness to (statistical) outliers in the data. Buchet et al.~\cite{Buchet2016} use this property to construct a filtration which is also provably stable in this sense. However, it is hard to compute the associated persistent homology, and they therefore employ an approximation scheme.

In this paper, we propose a novel filtration based on so-called \emph{Christoffel-Darboux kernels}. 
As we explain in more detail below, the resulting persistent homology can be computed in \emph{linear time} in the number of data points, and is provably robust to statistical outliers. Christoffel-Darboux (CD) kernels have a long history in fundamental mathematics, with applications to orthogonal polynomials and in approximation theory (see \cite{LPP2022} for an overview). 
They are the reproducing kernels $\CDkernel{\mu}{d} : \R^n \times \R^n \to \R$ for the Hilbert space $\R[\x]_d$ of $n$-variate, real polynomials of degree $d \in \N$ with respect to the inner product $\langle p, q \rangle = \int pq d\mu$ induced by a finite measure~$\mu$ on~$\R^n$. Such reproducing kernels completely describe the inner product $\langle \cdot, \cdot \rangle$, and in our setting the kernels $\CDkernel{\mu}{d}$ thus capture information about the underlying measure~$\mu$. For instance, the \emph{Christoffel polynomial} $\CDpol{\mu}{d}(\x) := \CDkernel{\mu}{d}(\x, \x)$ can be used to estimate the support $\mathrm{supp}(\mu) \subseteq \R^n$ of $\mu$. Roughly speaking, $\CDpol{\mu}{d}(\x)$ is small when $\x \in \mathrm{supp}(\mu)$ and large when $\x \not\in \mathrm{supp}(\mu)$ (see Proposition~\ref{PROP:christfunc} below).
This property has recently been applied to perform geometric inference in a statistical setting by Lasserre, Pauwels and Putinar~\cite{Lasserre2019, LPP2022, Pauwels2021}. 
In these works, the authors consider CD kernels for the \emph{empirical} measure~$\mu_\samples$ associated to a set of samples $\samples$ drawn according to some unknown measure~$\mu$. For fixed $d \in \N$, the polynomial $\CDpol{\mu_{\samples}}{d}$ associated to $\samples$ is straightforward to compute. Moreover, the \emph{sublevel set} $\{ \x \in \R^n : \CDpol{\mu_{\samples}}{d}(\x) \leq t\}$ captures the support of $\mu$ well for suitably selected~$t \geq 0$, see Figure~\ref{FIG:CDlevelsetexamples}. 
However, a key issue of this approach is that the level $t \geq 0$ must be selected `by hand' based on heuristics, and the quality of geometric inference depends heavily on this choice. This problem motivates our use of persistent homology, which  considers all sublevel sets simultaneously.

\begin{figure}
    \centering
    \begin{subfigure}[c]{0.24\textwidth}
    \centering
    \includegraphics[width=\textwidth,trim=2.2cm 1cm 2.2cm 1cm,clip]{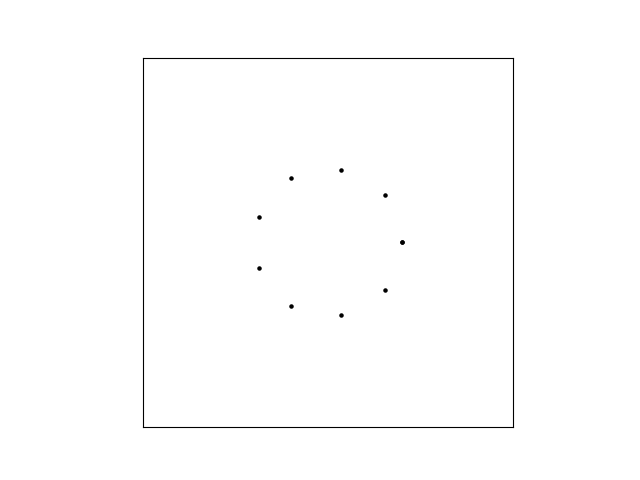}
    \caption{The samples $\samples$.}
    \end{subfigure}
    \hfill
    \begin{subfigure}[c]{0.24\textwidth}
        \centering
    \includegraphics[width=\textwidth,trim=2.2cm 1cm 2.2cm 1cm,clip]{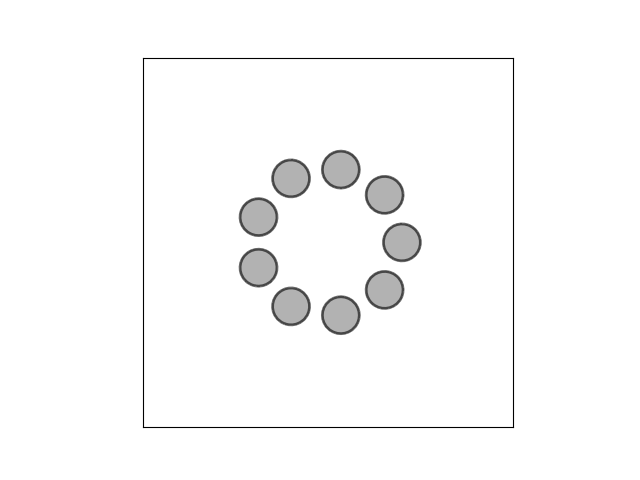}
        \caption{$\{ \x : d_\samples(\x) \leq 0.1\}$}
    \end{subfigure}
    \begin{subfigure}[c]{0.24\textwidth}
        \centering
    \includegraphics[width=\textwidth,trim=2.2cm 1cm 2.2cm 1cm,clip]{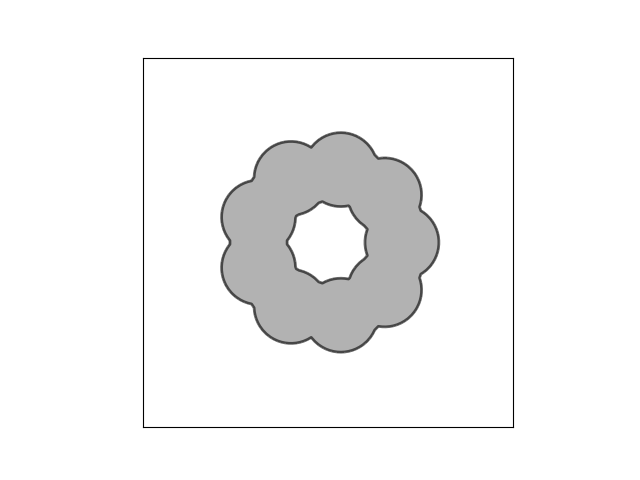}
        \caption{$\{ \x : d_\samples(\x) \leq 0.2 \}$}
    \end{subfigure}
    \hfill
    \begin{subfigure}[c]{0.24\textwidth}
        \centering
    \includegraphics[width=\textwidth,trim=2.2cm 1cm 2.2cm 1cm,clip]{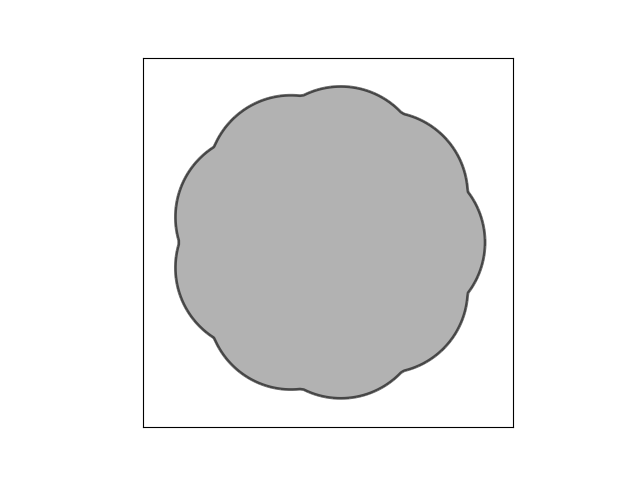}
        \caption{$\{ \x : d_\samples(\x) \leq 0.45 \}$}
    \end{subfigure}        \hfill
    \begin{subfigure}[b]{0.5\textwidth}
        \centering
        \begin{tikzpicture}[scale=2.5]
            \draw[-latex] (-.5,0) -- (2.6,0); \draw[dashed, lightgray] (0,0) -- (0,.8);
            \draw[dashed, lightgray] (1,0) -- (1,1.1);
            \draw[dashed, lightgray] (1.732,0) -- (1.732,1.1);
            \foreach \x in  {0,1,1.732}
            \node[blackpoint] at (\x,0){};
            % \foreach \x in  {0,1,2}
            % \node[textnode] at (\x,-.2) {\x};
            
            \node[textnode] at (1,-.2) {$t_1 \approx 0.2$};
            \node[textnode] at (1.732,-.2) {$t_2 = 0.4$};
            \node[textnode] at (2.5,-.2) {$t$};
            \node[whitepoint] (h01) at (0,.2){};
            \draw (0,.2) node[whitepoint] {} -- (2.6,.2);
            \foreach \x in  {.3,.4,...,1.0}
            \draw (0,\x) node[whitepoint] {} -- (1,\x) node[blackpoint] {};
            \draw (1,1.1) node[whitepoint] {} --  node[above]{$H_1$} (1.732,1.1) node[blackpoint] {};
            \draw (0,.1) -- ++(-.1,0) -- node[left]{$H_0$}  ++(0,1.0) -- ++(.1,0);
        \end{tikzpicture}
        \caption{The interval modules.}
    \end{subfigure}
    \hfill
    \begin{subfigure}[b]{0.35\textwidth}
        \centering
        \includegraphics[width=\textwidth,trim=2cm 0cm 3cm 0cm,clip]{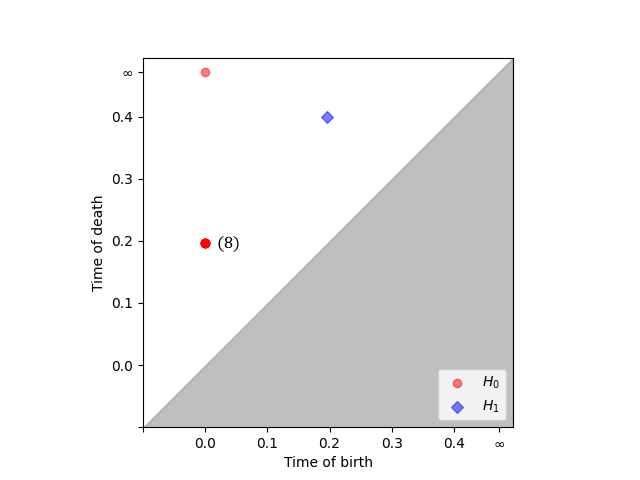}
        \caption{The persistence diagram.}
    \end{subfigure}
    \caption{A filtration of $[-1, 1]^2$ by the distance function $d_\samples : \x \mapsto 
    \mathrm{dist}(\x, \samples)$ to a set of equidistant points $\samples$ on a circle of radius $0.4$, and the corresponding persistence diagram. Note that there are $8$ intervals that are born at $t=0$ and die at $t \approx 0.2$, which show up as a single dot in the diagram. Throughout, we indicate the number of such overlapping dots in the diagram if necessary for clarity.}
    \label{FIG:persdia}
\end{figure}
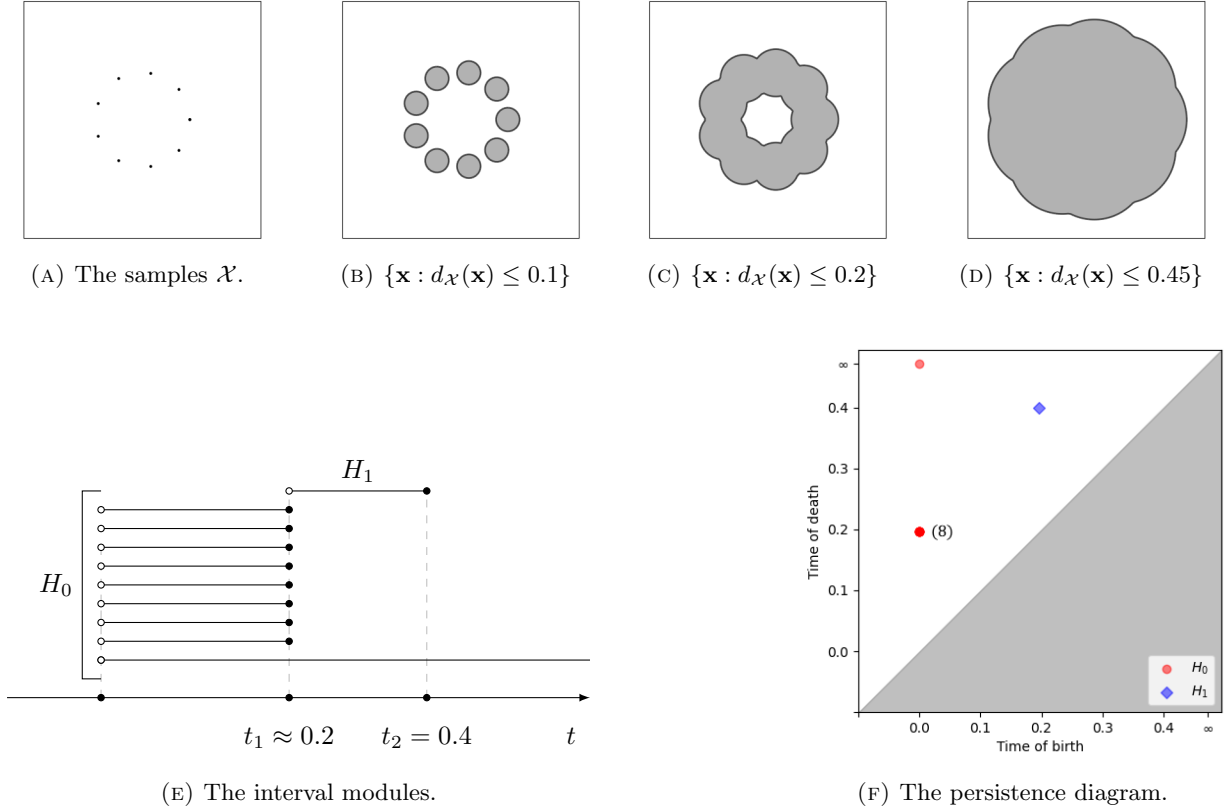
\begin{figure}
    \centering
    \includegraphics[width=0.4\textwidth,trim=2.2cm 0cm 2.2cm 0cm,clip]{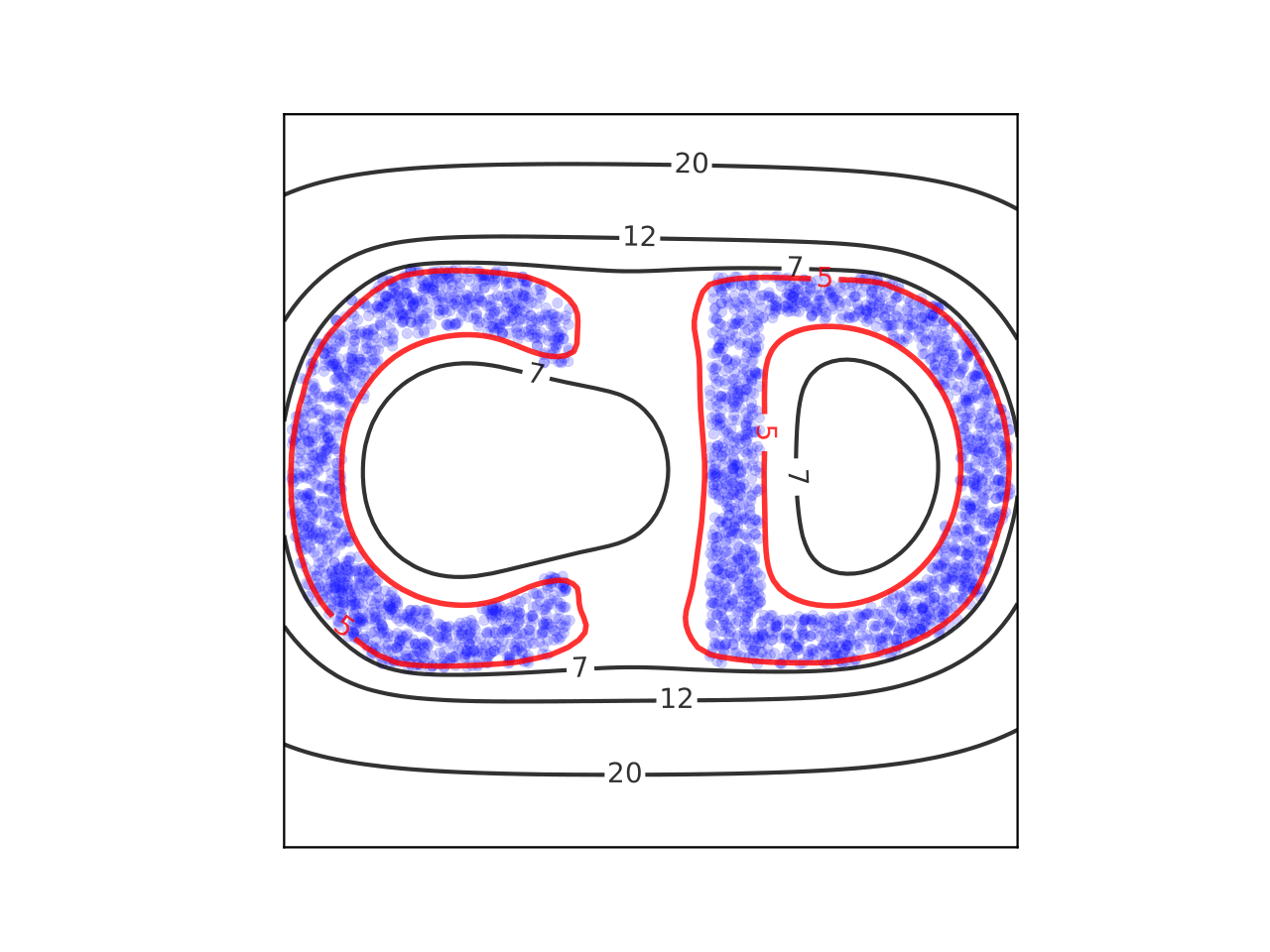}
    \includegraphics[width=0.4\textwidth,trim=2.2cm 0cm 2.2cm 0cm,clip]{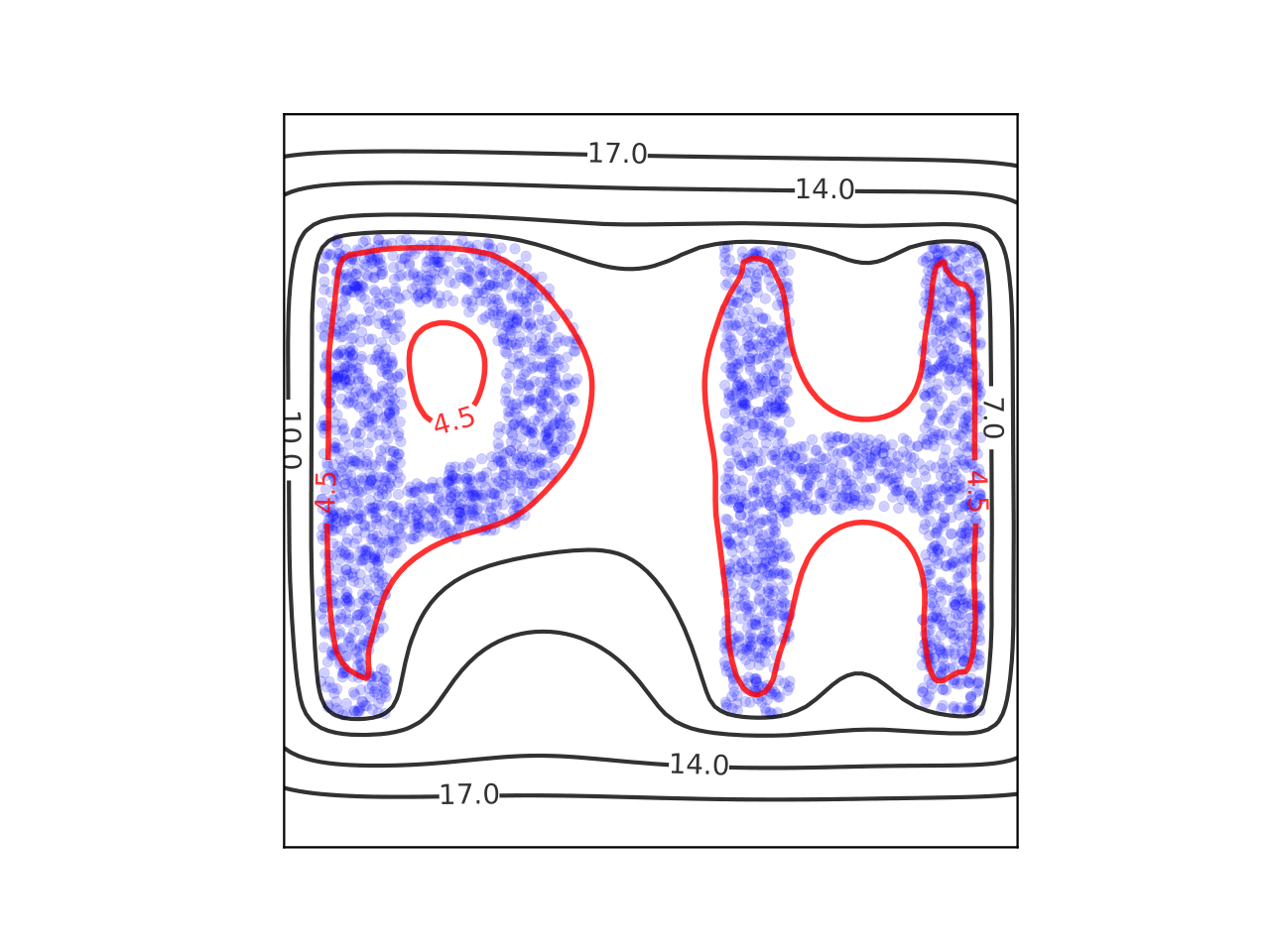}
        \caption{The level sets of the Christoffel polynomial $\x \mapsto \CDpol{\mu_{\samples}}{10}(\x)$ associated to the empirical measure $\mu_\samples$ of two sample sets $\samples \subseteq [-1, 1]^2$ (in blue). The level sets indicated in red capture the support of the underlying measure~$\mu$ quite well.}
    \label{FIG:CDlevelsetexamples}
\end{figure}

\subsection{Contributions and outline}
We propose a new scheme for topological data analysis of a finite point cloud $
\samples \subseteq [-1, 1]^n$, based on Christoffel-Darboux kernels. 
Our scheme unites recent applications of CD kernels in (statistical) data analysis with ideas from persistent homology. It consists of three steps:
\begin{enumerate}
    \item \textbf{Moment matrix}. Fix $d \in \N$. Choose a basis $\basis = (b_\alpha)$ for the space $\R[\x]_d$ of $n$-variate polynomials of degree at most~$d$. Compute the \emph{moment matrix} $M_d(\basis)$ of size ${n+d \choose d}$, whose entries can be computed from $\samples$ in linear time via:
    \[
        M_d(\basis)_{\alpha, \beta} := \frac{1}{|\samples|}\sum_{\x \in \samples} b_\alpha(\x) b_\beta(\x) \quad(\alpha, \beta \in \N^n,~|\alpha|, |\beta| \leq d).
    \]
    \item \textbf{Christoffel polynomial}. Invert the moment matrix to obtain the \emph{Christoffel polynomial}:
    \[
    \CDpol{}{d}(\x) := \basis(\x)^\top \big(M_d(\basis) \big)^{-1} \basis(\x),
    \]
    whose sublevel sets are known to approximate the set $\samples$, see Figure~\ref{FIG:CDlevelsetexamples}. 
    % $\mu_{\samples} := \frac{1}{|
% \samples|} \sum_{\x \in \samples} \delta_{\x}$.
    \item \textbf{Persistence module.} Define the \emph{sublevel set filtration}:
    \[
        \X_t := \{ \x \in [-1, 1]^n : \log \CDpol{}{d}(\x) \leq t\} \quad (t \geq 0),
    \]
    and compute its associated \emph{persistence module}: \[\persmod{d}{\samples} := \mathrm{\mathrm{PH}}_*([-1, 1]^n,~\log \CDpol{}{d}).
    \]
\end{enumerate}

\subsection*{Robustness to statistical outliers.} We show that the module $\persmod{d}{\samples}$ is stable and robust under perturbations of the input data $\samples$. To be precise, we show \emph{local} Lipschitz continuity of the function $\samples \mapsto \persmod{d}{\samples}$, in the \emph{Bottleneck} and \emph{Wasserstein} distance. We also give an estimate of the Lipschitz constant in terms of a concrete measure of algebraic degeneracy of the set $\samples$. This is our main technical result, see Section~\ref{SEC:stability}.

\subsection*{Exact algorithm with linear dependence on the number of samples.} We give an exact algorithm for computing the persistence module $\persmod{d}{\samples}$ in Section~\ref{SEC:algorithm}, whose runtime is linear in the number of data points, but depends exponentially on the dimension $n$. This algorithm is a combination of 1) a known procedure to compute CD kernels and 2) the recent work~\cite{Basu2022}, in which the authors propose an algorithm for computing the persistent homology of \emph{semialgebraic} filtrations. 

\subsection*{Numerical examples.} 
We provide several numerical examples in Section~\ref{SEC:experiments} that illustrate the geometric properties of our scheme, and its potential benefits and downsides compared to existing methods.
Unfortunately, there is no practical implementation available of the algorithm proposed in~\cite{Basu2022}. In order to perform numerical experiments, we therefore propose a simple scheme for approximating $\persmod{d}{\samples}$ in Section~\ref{SEC:approximation},  based on a triangulation of the sample space $[-1,1]^n$. These experiments show that our novel persistence module is able to accurately capture underlying homological features of point clouds, even in the presence of outliers.

\section{Background}

\subsection*{Notations and conventions.}
Throughout, $\x, \y, \z \in \R^n$ are $n$-dimensional variables. We denote by~$\R[\x]$ the $n$-variate polynomial ring. We write $\R[\x]_d \subseteq \R[\x]$ for the subspace of polynomials of (total) degree at most $d$, which has (real) dimension $s(n, d) = {n + d \choose d}$. 
% We denote by  $\|f\|_\infty$ = $\max_{\x \in [-1, 1]^n} |f(\x)|$  the supremum norm of a polynomial $f$ on~$[-1, 1]^n$. 
For ease of exposition, we assume throughout that sets of samples $\samples, \sampless$ are contained in the box $[-1, 1]^n$, which can always be achieved by a rescaling.

\subsection{Persistent homology}

Persistent homology is a central tool in topological data analysis, and has received a lot of attention over recent years \cite{persistencesurvey}. It serves to track homology classes through a diagram of spaces, typically arising from a filtration: 
Let $\X$ be a filtered topological space, that is, for each $t\in \R$, there is a subspace $\X_t \subseteq \X$, such that if $s\leq t$, $\X_s \subseteq \X_t$. For convenience, we assume that the filtration is exhaustive, i.e. $\bigcup_t \X_t = \X$. Applying homology with coefficients in $\mathbb{F}$ to each $\X_t$ then yields a diagram of spaces $\mathrm{H}_*(\X_t;\mathbb{F})$: For every $s\leq t$, the inclusion map $\iota_s^t: \X_s \to \X_t$ induces a map $h_s^t= (\iota_s^t)_*: \mathrm{H}_*(\X_s;\mathbb{F}) \to \mathrm{H}_*(\X_t;\mathbb{F})$, and the collection of maps $\{h_s^t\}$ satisfy: 
\begin{enumerate}
    \item For all $r\leq s \leq t$, $h_s^t \circ h_r^s = h_r^t$;
    \item For all $t \in \R$, $h_t^t = id_{\mathrm{H}_*(\X_t;\mathbb{F})}$.
\end{enumerate}
This diagram of spaces is the $\emph{persistent homology}$ of $\X$, denoted by $\mathrm{PH}_*(\X_t;\mathbb{F})$. Any {$\R$-indexed} collection of vector spaces with maps satisfying \textbf{1.} and \textbf{2.} above is called a \emph{persistence module}. More succinctly put, a persistence module is a functor from the poset $(\R,\leq)$ to the category of vector spaces over some field. The vector spaces can be taken over any field $\mathbb{F}$, but we always work over a finite field.

\begin{example}\label{EXM:filtration}
    Let $\X$ be a topological space, and let $f: \X \to \R$ be a continuous function. Then the sublevel set filtration of $\X$ with respect to $f$ is given by $\X_t = \{\x \in \X \, | \, f(\x) \leq t\}$. Applying homology to this filtration yields a persistence module, which we denote by $\mathrm{PH}_*(\X,f)$. 
\end{example}

If $\mathrm{H}_*(\X_t)$ is finite dimensional for each $t \in \R$ (which is a mild requirement, and will always be satisfied in our setting), then $\mathrm{PH}_*(\X)$ is completely described by a set of intervals, denoted by $\Dgm(\mathrm{PH}_*(\X))$. The presence of an interval $[t_b, t_d) \in \Dgm(\mathrm{PH}_p(\X))$ tells us that a particular $p$-dimensional homology class is born at time $t_b$, and lives until time $t_d$, where it then dies. If $t_d = \infty$, this means that this homology class lives forever, and corresponds to a global homology class in $\mathrm{H}_p(\X)$. The diagram $\Dgm(\mathrm{PH}_*(\X))$ can be conveniently visualized, see Figure~\ref{FIG:persdia}.

\subsubsection{Stability}
An important property of a persistence module one needs to verify before using it in an application, is that it is stable with respect to the input data. Intuitively, this means that small perturbations of the input data should result only in small perturbations in the obtained persistence diagrams. We make this precise below.

\begin{definition}
    A matching between two multi-sets $A$ and $B$ is a bijection $\chi$ between two subsets $A' \subset A$ and $B' \subset B$. We denote this by $\chi: A \nrightarrow B$. If $\chi$ matches $a \in A$ to $b \in B$, we write $(a,b) \in \chi$. If $c \in A \cup B$ is unmatched by $\chi$, we abuse notation and write $c \notin \chi$.
\end{definition}

\begin{definition}
    Let $I = \langle t_{b_1}, t_{d_1} \rangle$ and $J = \langle t_{b_2}, t_{d_2} \rangle $ be two intervals in $\overline{\R} = \R \cup \{-\infty, \infty\}$. Then the cost of $I$ is given by 
    \begin{align*}
        c(I) := (t_{d_1}-t_{b_1})/2
    \end{align*}
    and the cost of the pair $(I,J)$ is given by:
    \begin{align*}
        c(I,J) := \max\{|t_{b_1} - t_{b_2}|,~|t_{d_1} - t_{d_2}|\}
    \end{align*}
    Now let  $\mathcal{D}_1$ and $\mathcal{D}_2$ be two multi-sets of intervals in $\overline{\R} = \R \cup \{-\infty, \infty\}$. The cost of a matching $\chi : \mathcal{D}_1 \nrightarrow \mathcal{D}_2$ is defined as 
    \begin{align*}
        \mathrm{cost}(\chi) := \max\Bigl\{\sup_{(I,J) \in \chi} c(I,J),~\sup_{ I \notin \chi} c(I) \Bigr\}
    \end{align*}
    Finally, the \emph{Bottleneck distance} between $\mathcal{D}_1$ and $\mathcal{D}_2$ is given by:
    \begin{align*}
        d_B(\mathcal{D}_1,\mathcal{D}_2):= \inf_{\chi: \mathcal{D}_1 \nrightarrow \mathcal{D}_2} \mathrm{cost}(\chi)
    \end{align*}
\end{definition}
The Bottleneck distance is the most widely used distance on the space of persistence diagrams, and it satisfies the following:

\begin{theorem}[\cite{originalstability}]\label{COR:contstab} 
    Suppose $\X$ is a CW-complex, and $f,g: \X \to \R$ are two continuous functions on $\X$. Then
    \begin{align*}
        d_B(\Dgm(\mathrm{PH}_p(\X,f)), \Dgm(\mathrm{PH}_p(\X,g))) \leq \| f-g\|_\infty := \max_{\x \in \X} |f(\x) - g(\x)|.
    \end{align*}
\end{theorem}

\begin{definition}
    Let $\X$ and $\Y$ be two subsets of a metric space $(M,d)$. Write $U_\varepsilon(\X) := {\{ m \in M \, | \, \dist(m,\X) \leq \varepsilon\}}$. Then the Hausdorff distance between $\X$ and $\Y$ is given by:
    \begin{align*}
        d_H(\X,\Y) :&= \inf\{\varepsilon \geq 0 \, | \, \Y \subseteq U_\varepsilon(\X)  \text{ and } \X \subseteq U_\varepsilon(Y)\}\\
        &= \max \Bigl\{ \sup_{\x\in \X}\dist(\x,\Y),~ \sup_{\y\in \Y}\dist(\y,\X) \Bigr\}.
    \end{align*}
\end{definition}

\begin{example}
    Let $\X \subseteq M,$ be a subset of a metric space $(M,d)$. Then the distance function 
    \[
    d_\X: M \to \R, \quad \x \mapsto \dist(\x, \X)
    \]
    is a continuous function on $M$, and defines a sublevel set filtration and a persistence module, which we suggestively denote by $\mathrm{PH}_p(\check{C}(\X))$. Note that if $\Y$ is another subset of $M$, then $\|d_\X - d_\Y\|_\infty = d_H(\X,\Y)$, so it follows from Theorem \ref{COR:contstab} that  
    \begin{align*}
        d_B\left(\Dgm(\mathrm{PH}_p(\check{C}(\X))),~ \Dgm(\mathrm{PH}_p(\check{C}(\Y)))\right) \leq d_H(\X,\Y).
    \end{align*}
\end{example}
    In the above example, $\X$ is typically a finite subset of $\mathbb{R}^n$, and this is often used as one of the motivating examples for persistent homology. When $\X$ is sampled from some unknown shape $\mathfrak{X}$ inside of $\mathbb{R}^n$, its persistent homology can be used to estimate the homology of $\mathfrak{X}$. It can be computed using the \v{C}ech filtration of $\X$, which is a filtered simplicial complex whose homotopy type at each stage agrees with that of the sublevel set of $d_\X$ at the same scale. It is true, but not entirely straight-forward, that the two persistence modules arising from these constructions are isomorphic \cite{unifiednervepreprint, functnervesteve}.

\subsubsection{Wasserstein distance}
In Section~\ref{SEC:stability}, we will show stability results for our novel persistence module in terms of the \emph{Wasserstein distance}.
The Wasserstein distance is a metric on the space of probability measures supported on $\R^n$. It is commonly used in the context of optimal transport and (statistical) data analysis, see, e.g.,~\cite{OptimalTransport}. The primary advantage of the Wasserstein distance over the Hausdorff distance is that it is much less sensitive to outliers, and therefore more suited to applications in statistics. For our purposes, it is enough to consider probability measures with \emph{finite support}. 

\begin{definition}
Let $\mu_\samples, \mu_\sampless$ be two probability measures with finite supports $\samples, \sampless \subseteq~\R^n$, respectively. The Wasserstein distance $d_W(\mu_\samples, \mu_\sampless)$ is then given by the optimum solution to the linear program:
\begin{alignat}{2} \label{EQ:Wasserstein}
    d_W(\mu_\samples, \mu_\sampless) := \min_{\gamma}& \quad\quad&&\smashoperator[l]{\sum_{\x \in \samples, \y \in \sampless}} \gamma(\x, \y) \cdot \|\x - \y\|_2 \\
    \rm{s.t.}& &&\sum_{\y \in \sampless} \gamma(\x, \y) = \mu_{\samples}(\x) \\
    & &&\sum_{\x \in \samples} \gamma(\x, \y) = \mu_{\sampless}(\y) \\
    & &&\gamma : \samples \times \sampless \to \R_{\geq 0}.
\end{alignat}
\end{definition}
One can think of $d_W(\mu_\samples, \mu_\sampless)$ as the amount of `work' required to transform the measure $\mu_\samples$ into $\mu_\sampless$. For instance, if $\sampless = \{\y\}$ is a singleton, then $d_W(\mu_\samples, \mu_{\{\y\}})$ is simply given by:
\[
    d_W(\mu_\samples, \mu_{\{\y\}}) = \sum_{\x \in \samples} \mu_\samples(\x) \cdot \|\x - \y\|_2.
\]

\subsection{The Christoffel-Darboux kernel} 
\label{SEC:CD}
In this section, we introduce some basic facts on Christoffel-Darboux kernels, with emphasis on the statistical setting. We refer to the book of Lasserre, Pauwels and Putinar~\cite{LPP2022} for a comprehensive treatment.
Let $\mu$ be a finite, positive Borel measure on $\R^n$ with compact, full-dimensional support. (In our setting, it is helpful to think of $\mu$ as the restriction of the Lebesgue measure to a sufficiently nice compact subset of $\R^n$). Then $\mu$ induces an inner product on the space $\R[\x]$ of $n$-variate, real polynomials via:
\begin{equation} \label{EQ:innerprod}
    \langle p, q \rangle_{\mu} := \int p(\x) q(\x) d\mu(\x).
\end{equation}
We can choose an ortho\emph{normal} basis $\basis = \{ b_\alpha : \alpha \in \N^n \}$ for $\R[\x]$ with respect to $\langle \cdot, \cdot \rangle_{\mu}$, which we order so that $b_\alpha \in \R[\x]$ is of total degree $|\alpha| = \sum_{i=1}^n \alpha_i$ for each $\alpha \in \N^n$. That is, we have the orthogonality relations:
\begin{equation} \label{EQ:orthonorm}
     \langle b_\alpha, b_\beta \rangle_{\mu} = \int b_\alpha(\x) b_\beta(\x) d\mu(\x) = \delta_{\alpha\beta} \quad (\alpha, \beta \in \N^n).
\end{equation}
Using this orthonormal basis, we can define the Christoffel-Darboux kernel.
\begin{definition}
For $d \in \N$, the \emph{Christoffel-Darboux kernel} $\CDkernel{\mu}{d} : \R^n \times \R^n \to \R$ of degree $d$ for the measure $\mu$ is defined as:
\begin{equation}
    \CDkernel{\mu}{d}(\x, \y) := \sum_{|\alpha| \leq d} b_\alpha(\x) b_\alpha(\y).
\end{equation}
\end{definition}
The Christoffel-Darboux kernel $\CDkernel{\mu}{d}$ is also called the \emph{reproducing kernel} for the Hilbert space $(\R[\x]_d, \langle \cdot, \cdot \rangle_{\mu})$, as we have the reproducing property:
\begin{equation} \label{EQ:reproducing}
    \int \CDkernel{\mu}{d}(\x, \y) p(\y) d\mu(\y) = \langle \CDkernel{\mu}{d}(\x, \cdot), p(\cdot) \rangle_{\mu} = p(\x) \quad(p \in \R[\x]_d).
\end{equation}
% Indeed, after decomposing $p = \sum_{|\alpha| \leq d} \lambda_\alpha b_\alpha$ in the orthonormal basis, we find that:
% \[
%     \int \CDkernel{\mu}{d}(\x, \y) p(\y) d\mu(\y) = \sum_{|\alpha| \leq d} \sum_{|\beta| \leq d} \lambda_\alpha \int b_\beta(\x) b_\beta(\y) b_\alpha(\y) d\mu(\y) = \sum_{|\alpha| \leq d} \lambda_\alpha b_\alpha(\x).
% \]
We note that the kernel $\CDkernel{\mu}{d}$ is independent of our choice of basis $\{b_\alpha\}$, and it can be computed via the Gram-Schmidt procedure even if we do not have access to an explicit orthonormal basis.
\begin{proposition}[Gram-Schmidt, see Prop. 4.1.2 in~\cite{LPP2022}] \label{PROP:matrixinverse}
Let $d \in \N$ and let $\basis = \{ b_\alpha : |\alpha| \leq d\}$ be \emph{any} basis for $\R[\x]_d$. For $\x \in \R^n$, write ${\basis_d(\x) = (b_\alpha (\x))_{|\alpha| \leq d}} \in \R^{s(n, d)}$ and consider the matrix $M^\mu_d(\basis) \in \R^{s(n, d) \times s(n, d)}$ given by the entrywise integral:
\begin{equation} \label{EQ:moment_matrix}
\begin{split}
        M^\mu_d(\basis) &:= \int \basis_d(\x) \basis_d(\x)^\top d\mu(\x),
        \\
        \text{i.e., }\quad (M^\mu_d(\basis))_{\alpha, \beta} &= \int b_\alpha(\x) b_\beta(\x) d\mu(\x) \quad (\alpha, \beta \in \N^n_d).
\end{split}
\end{equation}
The matrix $M^\mu_d(\basis)$ is strictly positive semidefinite, i.e., its eigenvalues are all stricly larger than $0$. 
Moreover, we have:
\[
\CDkernel{\mu}{d}(\x, \y) = \basis_d(\x)^\top \big(M^\mu_d(\basis)\big)^{-1} \basis_d(\y).
\]
\end{proposition}

\subsection*{The Christoffel polynomial}
For our purposes, we are mostly interested in the \emph{Christoffel polynomial} 
$\CDpol{\mu}{d} : \R^n \to \R$, defined in terms of an orthonormal basis~$\basis$ for $(\R[\x]_d, \langle \cdot, \cdot \rangle_{\mu})$ as:
\begin{equation} \label{EQ:christpol}
    \CDpol{\mu}{d}(\x) := \CDkernel{\mu}{d}(\x, \x) = \sum_{|\alpha| \leq d} b_\alpha(\x)^2.
\end{equation}
The Christoffel polynomial is a \emph{sum of squares} of polynomials, implying immediately that $\CDpol{\mu}{d}(\x) \geq 0$ for all $\x \in \R^n$. In fact, by definiteness of the inner product~\eqref{EQ:innerprod}, it is strictly positive on $\R^n$. It has a remarkable alternative definition in terms of a \emph{variational problem}:
\[
    \frac{1}{\CDpol{\mu}{d}(\z)} = \min_{p \in \R[\x]_d} \left\{ \int p^2(\x) d\mu(\x) : p(\z) = 1 \right\} \quad (\z \in \R^n).
\]
The Christoffel polynomial encodes information on the support of $\mu$. Roughly speaking, $\CDpol{\mu}{d}(\x)$ is rather \emph{large} when $\x \not\in \mathrm{supp}(\mu)$, and rather \emph{small} when ${\x \in \mathrm{supp}(\mu)}$ (see Figure~\ref{FIG:CDlevelsetexamples}). This can be made precise in the regime~$d \to \infty$. 
\begin{proposition}[see Sec.~4.3 of~\cite{LPP2022}] \label{PROP:christfunc}
Under certain assumptions on the measure $\mu$, we have:
\[
    \lim_{d \to \infty} \CDpol{\mu}{d}(\x) = \begin{cases} 
        O(d^n) &\x \in \mathrm{int}\big(\mathrm{supp}(\mu)\big), \\
        \Omega(\exp(\alpha d)) & \x \not\in \mathrm{supp}(\mu),
    \end{cases} 
\]
for any fixed $\x \in \R^n$. Here, the constant $\alpha$ is proportional to $\dist(\x, \mathrm{supp}(\mu))$. 
\end{proposition}

When $\mu$ is the restriction of the Lebesgue measure to a sufficiently nice compact subset $\mathfrak{X} \subseteq \R^n$, a stronger result is shown by Lasserre and Pauwels~\cite{Lasserre2019}.
\begin{theorem}[reformulation of Thm.~7.3.2 in \cite{LPP2022}]
\label{THM:CDconvergence}
Let $\mathfrak{X} \subseteq \R^n$ be a compact set satisfying the conditions of Assumption~7.3.1 in \cite{LPP2022}, and let $\mu$ be the restriction of the Lebesgue measure to $\mathfrak{X}$. Then there exist sequences $(t_k)_{k \in \N}$ and $(d_k)_{k \in \N}$ such that the sublevel sets
$\X_k := \{ \x \in \R^n : \CDpol{\mu}{d_k}(\x) \leq t_k \}
$ satisfy:
% \]
% satisfy:
% \begin{align*}
\[
        \lim_{k \to \infty} d_{H}(\X_k, \mathfrak{X}) = 0, \quad \text{and} \quad
        \lim_{k \to \infty} d_{H}(\partial \X_k, \partial \mathfrak{X}) = 0.
\]
% \end{align*}
Here, $\partial \mathfrak{X}$ and  $\partial \X_k$ denote the boundary of $\mathfrak{X}$ and $\X_k$, respectively.
\end{theorem}

\subsubsection{The empirical setting}
\label{SEC:empirical}
Assume now that we do not have explicit knowledge of the measure $\mu$, but are instead given a sequence $\samples \subseteq \R^n$ of $N$ samples $X_{1}, X_2, \dots, X_{N} \in \R^n$, drawn independently according to~$\mu$.
These samples induce a probability measure $\mu_{\samples}$ given by $\mu_{\samples} = \frac{1}{N}\sum_{i=1}^N \delta_{X_i}$, which we call the emperical measure associated to $\samples$. Under a non-degeneracy assumption (Assumption~\ref{ASSU:nondegen} below), the measure~$\mu_{\samples}$ induces an inner product of the form~\eqref{EQ:innerprod} on $\R[\x]_d$ by: 
\begin{equation} \label{EQ:innerprodemp}
    \langle p, q \rangle_{\mu_{\samples}} := \int p(\x) q(\x) d\mu_{\samples}(\x) = \frac{1}{N} \sum_{i=1}^N p(X_i) q(X_i).
\end{equation}
In light of~\eqref{EQ:innerprodemp} and Proposition~\ref{PROP:matrixinverse}, it is straightforward to compute the Christoffel-Darboux kernel~$\CDkernel{\mu_{\samples}}{d}$ of degree $d$ for the measure $\mu_{\samples}$ (and thus to compute $\CDpol{\mu_{\samples}}{d}$). Indeed, the entries of the matrix $M^{\mu_{\samples}}_d(\basis)$ in~\eqref{EQ:moment_matrix} may each be computed in time $O(N)$, after which $M^{\mu_{\samples}}_d(\basis)$ can be inverted in time $O(s(n, d)^3)$.
\begin{proposition} \label{PROP:CDcomplexity}
The empirical Christoffel-Darboux kernel $\CDkernel{\mu_{\samples}}{d}$ of degree $d$ for $N$ samples in $\R^n$ may be computed in time $O(Ns(n, d)^2 + s(n, d)^3)$.
\end{proposition}
The above procedure only works when the matrix $M^{\mu_{\samples}}_d(\basis)$ is invertible, or equivalently, when the `inner product' $\langle \cdot, \cdot \rangle_{\mu_{\samples}}$ on $\R[\x]_d$ is definite. We shall make this assumption throughout. 
\begin{assumption} \label{ASSU:nondegen}
We say a sample set $\samples$ is non-degenerate (up to degree $d$) if the inner product $\langle \cdot, \cdot \rangle_{\mu_{\samples}}$ associated to the empirical measure $\mu_{\samples}$ induced by $\samples$ via~\eqref{EQ:innerprodemp} is definite for polynomials up to degree $d$.
% Equivalently, this means that the matrix~$M^{\mu_{\samples}}_d$ is non-singular (for any choice of basis).
\end{assumption}
Assumption~\ref{ASSU:nondegen} is satisfied if and only if the samples $\samples$ are not contained in an algebraic hypersurface of degree $d$ (i.e., the zero set of a polynomial of degree at most $d$). This implies in particular that $N \geq s(n, d) + 1$.

Under certain assumptions on $\mu$, Lasserre and Pauwels~\cite{Lasserre2019} show that the (emperical) Christoffel polynomial $\CDpol{\mu_{\samples}}{d}$ converges to the (population) Christoffel polynomial $\CDpol{\mu}{d}$ as the number of samples $N \to \infty$. The rate of this convergence can be quantified, see~\cite{MaiTrang2022}.

\begin{theorem}[\cite{Lasserre2019}, see also \cite{LPP2022}]
\label{THM:lawlargenumbers}
Let $\samples = (X_1, X_2, \dots, X_N)$ be sampled from~$\mu$ as in the above. Then for each $\x \in [-1, 1]^n$, we have
$
    \lim_{N \to \infty} | \CDpol{\mu}{d}(\x) - \CDpol{\mu_\samples}{d}(\x)| = 0 \text{ almost surely}.
$
\end{theorem}

\section{A persistence module based on the Christoffel polynomial}
Theorems~\ref{THM:CDconvergence} and~\ref{THM:lawlargenumbers} motivate the use of the Christoffel polynomial in (statistical) data analysis. They show that certain sublevel sets of the empirical Christoffel polynomial approximate the support of the underlying population measure $\mu$ well (in Hausdorff distance) as the number of samples grows. However, Theorem~\ref{THM:CDconvergence} gives very little explicit information on \emph{which} (sub)level set to consider. This is the primary motivation for considering a persistent scheme instead, which we introduce now.

\begin{definition}
Fix $d \in \N$, and let $\samples \subseteq [-1, 1]^n$ be a set of samples whose associated empirical measure $\mu_\samples$ satisfies Assumption~\ref{ASSU:nondegen}. Let $\CDpol{\mu_{\samples}}{d} : \R^n \to \R$ be the corresponding Christoffel polynomial~\eqref{EQ:christpol}. For $t \geq 0$, we consider the compact sublevel set
\begin{equation} \label{EQ:sublevel}
       \X_{t} := \{ \x \in [-1,1]^n : \log \CDpol{\mu_{\samples}}{d}(\x) \leq t\}  = \{ \x \in [-1,1]^n : \CDpol{\mu_{\samples}}{d}(\x) \leq 10^t\},
\end{equation}
which is well-defined as $\CDpol{\mu_{\samples}}{d}(\x) \geq 1$ for all $\x \in \R^n$.
By definition $(\X_t)_{t \geq 0}$ is a filtration, i.e, $\X_{t} \subseteq \X_{t'}$ for any $t \leq t'$. In light of Example~\ref{EXM:filtration}, we can therefore define the persistence module
\begin{equation}
    \persmod{d}{\samples} := \mathrm{PH}_*([-1, 1]^n, ~\log \CDpol{\mu_\samples}{d}).
\end{equation}
\end{definition}
Notably, we do not consider the level sets of $\CDpol{\mu_{\samples}}{d}$, but rather those of $\log \CDpol{\mu_{\samples}}{d}$. Before we motivate this choice, let us first observe that from a computational perspective, this logarithmic rescaling makes no difference. Indeed, 
one can obtain the persistence module of the filtration $([-1,1]^n, \log \CDpol{\mu_{\samples}}{d})$ by first computing the module associated to $([-1,1]^n, \CDpol{\mu_{\samples}}{d})$ and then rescaling all interval modules. We choose to work with $\log \CDpol{\mu_\samples}{d}$ for two reasons. First, as we will see below, this choice allows us to prove a stronger and more elegant stability result for the module $\persmod{d}{\samples}$. Second, the logarithmic scaling produces persistence diagrams that better fit the underlying topology in practice. Proposition~\ref{PROP:christfunc} provides a rather convincing theoretical argument for this observation. Indeed, if $\x \in [-1, 1]^n$ is a point outside of the support of the underlying measure $\mu$, it tells us that 
    $
        \CDpol{\mu}{d}(\x) \approx \exp \big( \dist (\x, \mathrm{supp}(\mu)) \cdot d\big),
    $
    which is to say that $\log \CDpol{\mu}{d}(\x)$ scales \emph{linearly} in the distance $\dist(\x, \mathrm{supp}(\mu))$, an intuitively desirable property.
    
\subsection{Stability and robustness}
\label{SEC:stability}

In this section, we show that the module $\persmod{d}{\samples}$ is \emph{locally} stable under small perturbations of the sample set $\samples$, measured in the Wasserstein distance~\eqref{EQ:Wasserstein}. Namely, we show in Proposition~\ref{PROP:localLip}  that 
\[ d_B\big(\mathrm{Dgm}(\persmod{d}{\samples}),~\mathrm{Dgm}(\persmod{d}{\sampless})\big) \leq  \log \big (C_{\samples} \cdot d_W(\mu_\samples, \mu_\sampless) + 1 \big )
\]
for fixed $\samples \subseteq [-1, 1]^n$ and any $\sampless \subseteq [-1, 1]^n$ for which $d_W(\mu_\samples, \mu_\sampless)$ is sufficiently small. Here, the constant $C_\samples$ depends on $n, d$, and
the supremum norm
\[
    \| \log \CDpol{\mu_{\samples}}{d} \|_\infty := \max_{\x \in [-1, 1]^n} |\log \CDpol{\mu_{\samples}}{d}(\x)|,
\]
which we interpret as a `measure of algebraic degeneracy' of the set $\samples$ (see Assumption~\ref{ASSU:nondegen}).
For \emph{arbitrary} sets $\samples$ and $\sampless$, we show in Proposition~\ref{PROP:logstab} that 
\begin{equation} \label{EQ:dBXY}
d_B\big(\mathrm{Dgm}(\persmod{d}{\samples}),~\mathrm{Dgm}(\persmod{d}{\sampless})\big) \leq \log \big (C_{\samples, \sampless} \cdot d_W(\mu_\samples, \mu_\sampless) + 1 \big ),
\end{equation}
where the constant $C_{\samples, \sampless}$ now additionally depends on $\| \log \CDpol{\mu_{\sampless}}{d} \|_\infty$. If one restricts to `sufficiently non-degenerate' sample sets (i.e., those having $\|\log \CDpol{\mu_{\samples}}{d} \|_\infty$ bounded from above), relation~\eqref{EQ:dBXY} may be read as a \emph{global} stability result.

For the proof of these statements, note first that in light of Corollary~\ref{COR:contstab}, we have:
\begin{equation} \label{EQ:supnorm}
d_B\big(\mathrm{Dgm}(\persmod{d}{\samples}),~\mathrm{Dgm}(\persmod{d}{\sampless})\big) \leq 
   \|\log \CDpol{\mu_{\samples}}{d} - \log \CDpol{\mu_{\sampless}}{d}\|_\infty 
%   := \max_{\x \in [-1, 1]^n} |\CDpol{\mu_{\samples}}{d}(\x) - \CDpol{\mu_{\sampless}}{d}(\x)|
\end{equation}
and so it suffices to consider the quantity $\|\log \CDpol{\mu_{\samples}}{d} - \log \CDpol{\mu_{\sampless}}{d}\|_\infty$. We start by showing the following.

\begin{theorem} \label{THM:wasstab}
Let $\samples, \sampless \subseteq \R^n$ be as in the above. Write $C_{n, d} := 4 \cdot s(n, d) \cdot d^2$, where $s(n,d) := \dim \R[\x]_d = {n + d \choose d}$. Then for all $\x \in \R^n$, we have that:
\[
    |\CDpol{\mu_{\samples}}{d}(\x) - \CDpol{\mu_{\sampless}}{d}(\x)| \leq C_{n, d} \cdot \|\CDpol{\mu_\samples}{d}\|_{\infty} \cdot d_W(\mu_\samples, \mu_{\sampless}) \cdot \CDpol{\mu_\sampless}{d}(\x).
\]
\end{theorem}
We prove Theorem~\ref{THM:wasstab} in Appendix~\ref{APP:technical} by adapting the techniques of Section~6.2 in~\cite{LPP2022} to our setting. This requires some small technical statements on Wasserstein distance and supremum norms of polynomials on compact domains, which we also give there. 

We have the following immediate consequence:
\begin{corollary} \label{COR:divstab}
Let $\samples, \sampless$ as in the above, and let $C_{n, d} > 0$ be the constant of Theorem~\ref{THM:wasstab}. Then for all $\x \in [-1,1]^n$, we have:
\[
    |\CDpol{\mu_{\samples}}{d}(\x) / \CDpol{\mu_{\sampless}}{d}(\x) - 1| \leq C_{n, d} \cdot \|\CDpol{\mu_\samples}{d}\|_{\infty} \cdot d_W(\mu_\samples, \mu_{\sampless}).
\]  
\end{corollary}
After taking logarithms (see Appendix~\ref{APP:technical} for details), we then obtain:
\begin{proposition} \label{PROP:logstab}
    Let $\samples,\sampless \subseteq \R^n$ be as above, and let $C_{n,d}>0$ be the constant of Theorem~\ref{THM:wasstab}. Then we have:
    \begin{align} \label{EQ:logstab}
         \|\log\CDpol{\mu_\samples}{d} - \log\CDpol{\mu_\sampless}{d}\|_\infty &\leq \log \big(C_{n,d} \cdot \max\{\|\CDpol{\mu_\samples}{d}\|_\infty, \|\CDpol{\mu_\sampless}{d}\|_\infty\} \cdot d_W(\mu_\samples,\mu_\sampless) + 1 \big). 
        %  \\
        %  &\leq C_{n,d} \cdot \max\{\|\CDpol{\mu_\samples}{d}\|_\infty, \|\CDpol{\mu_\sampless}{d}\|_\infty\} \cdot d_W(\mu_\samples,\mu_\sampless).
        %  %\\
         %&\approx 
    %\begin{cases}
     %   C_{n,d} \cdot \max\{\|\CDpol{\mu_\samples}{d}\|_\infty, \|\CDpol{\mu_\sampless}{d}\|_\infty\} \cdot d_W(\mu_\samples,\mu_\sampless) & \text{for small } d_W(\mu_\samples,\mu_\sampless) \\
     %   \log(C_{n,d}) + \log(\max\{\|\CDpol{\mu_\samples}{d}\|_\infty, \|\CDpol{\mu_\sampless}{d}\|_\infty\}) + \log(d_W(\mu_\samples,\mu_\sampless)), & \text{else }  .
     %   \end{cases}
    \end{align}
\end{proposition}

Note that when the quantity $C := C_{n,d} \cdot \max\{\|\CDpol{\mu_\samples}{d}\|_\infty, \|\CDpol{\mu_\sampless}{d}\|_\infty\} $ is close to $0$, the bound~\eqref{EQ:logstab} tells us that
$
\|\log\CDpol{\mu_\samples}{d} - \log\CDpol{\mu_\sampless}{d} \|_\infty \leq C \cdot d_W(\mu_\samples,\mu_\sampless)
$.
On the other hand, if $C \gg 1$, 
% $C_{n,d} \cdot \max\{\|\CDpol{\mu_\samples}{d}\|_\infty, \|\CDpol{\mu_\sampless}{d}\|_\infty\} \cdot d_W(\mu_\samples,\mu_\sampless)$ 
the bound on $\|\log(\CDpol{\mu_\samples}{d}) - \log(\CDpol{\mu_\sampless}{d} )\|_\infty$ is exponentially smaller than in Theorem~\ref{THM:wasstab}.
Finally, we may get rid of the dependence on $\sampless$ in 
Corollary~\ref{COR:divstab} after making an additional assumption.
\begin{proposition} \label{PROP:localLip}
Let $\samples \subseteq [-1,1]^n$ be a (fixed) sample set satisfying Assumption~\ref{ASSU:nondegen}. Let $C_{n,d} > 0$ be the constant of Theorem~\ref{THM:wasstab}, and assume that $\sampless \subseteq [-1, 1]^n$ is such that $C_{n,d} \cdot \|\CDpol{\mu_{\samples}}{d}\|_\infty \cdot d_W(\mu_\samples, \mu_\sampless)  \leq 1/2$. Then we have $\|\CDpol{\mu_\sampless}{d} \|_\infty \leq 2 \|\CDpol{\mu_\samples}{d}\|_{\infty}$.
In particular, the bound~\eqref{EQ:logstab} then reads:
\[
\|\log\CDpol{\mu_\samples}{d} - \log\CDpol{\mu_\sampless}{d} \|_\infty \leq \log \big(C_{n,d} \cdot 2\|\CDpol{\mu_\samples}{d}\|_\infty \cdot d_W(\mu_\samples,\mu_\sampless) + 1 \big).
\]
\end{proposition}

\subsection{An exact algorithm for the persistence module} \label{SEC:algorithm}
As we explain in Section~\ref{SEC:CD} (see Proposition~\ref{PROP:CDcomplexity}), the Christoffel polynomial $\CDpol{\mu_{\samples}}{d}$ of degree $d$ may be computed in time $O(s(n,d)^3 + N s(n,d)^2)$, where $N = |\samples|$ is the number of samples and $s(n,d) = {n+d\choose d}$ is the size of the moment matrix~\eqref{EQ:moment_matrix}. Once we have access to~$\CDpol{\mu_{\samples}}{d}$, it remains to compute the persistent homology of the filtration $\{\x \in [-1, 1]^n : \CDpol{\mu_{\samples}}{d}(\x) \leq t\}_{t \geq 0}$. The set $[-1, 1]^n$ is a particularly simple example of a basic (closed) \emph{semialgebraic set}. That is, a subset of $\R^{n}$ defined by a finite number of polynomial (in)equalities; namely:
\[
    [-1, 1]^n = \{ \x \in \R^n : g_i(\x) := 1-\x_i^2 \geq 0~\text{ for } i=1, 2, \ldots, n\}.
\]
We may therefore use the following recent result of Basu and Karisani~\cite{Basu2022}.
\begin{theorem}[Basu, Karisani (2022)] \label{Basu}
For fixed $p \in \N$, there is an algorithm that takes
as input a description of a closed and bounded semialgebraic set $ \X = \{\x \in \R^n : {g_i(\x) \geq 0}, \ {i=1,2,\dots,m}\}$, and a polynomial $P \in \R[\x]$, and outputs the persistence diagram associated to the filtration $\{ \x \in \X : P(\x) \leq t \}_{t \geq 0}$ up to dimension $p$. The complexity of this algorithm is bounded by $(md)^{O(n)}$, where $d = \max \{ \mathrm{deg}(g_i), \mathrm{deg} (P)\}$ is the largest degree amongst $P$ and the polynomial inequalities defining $\X$.
\end{theorem}

\begin{corollary}[Exact algorithm]
For fixed $p \in \N$, there is an exact algorithm that computes the persistence diagram associated to $\persmod{d}{\samples}$ up to dimension $p$. Its runtime is bounded by:
\[
    O(s(n,d)^3 + N s(n,d)^2) + (nd)^{O(n)}.
\]
\end{corollary}

\newcommand{\triang}{\mathcal{K}}
\newcommand{\approxf}{\overline{f}}
\subsection{An effective approximation scheme}
\label{SEC:approximation}

To the authors' knowledge, no implementation exists of the algorithm mentioned in Theorem~\ref{Basu} at the time of writing. In order to perform numerical experiments, we use a simple approximation scheme for $\persmod{d}{\samples}$. This method works in the more general case of approximating $\mathrm{PH}_*([-1,1]^n,f)$ for any Lipschitz continuous function $f: [-1,1]^n \to \R$. Succinctly put, we first fix $m \in \N$, and construct the Freudenthal triangulation~\cite{Freudenthal,eaves1984course} of $[-1,1]^n$, with vertices equal to the lattice points of $\frac{2}{m} \cdot \mathbb{Z}^n$ contained in $[-1,1]^n$. See Figure~\ref{FIG:Freudenthal}. We denote this triangulation by $\triang_m$. Note that the diameter of any simplex in this triangulation is equal to $2\sqrt{n}/m$. We then evaluate $f$ on each of the vertices, and compute the persistent homology of the lower-star filtration on $\triang_m$ induced by these function values. This persistence module, denoted by by $\mathrm{PH}_*(\triang_m,f)$, approximates $\mathrm{PH}_*([-1,1]^n,f)$. The following proposition gives a guarantee on the quality of this approximation:

\begin{proposition}\label{PROP:Lipschitzstability}
    Let $f: [-1,1]^n \to \R$ be a Lipschitz continuous function with Lipschitz constant $L_f$, choose $m\in \N$, and let $\triang_m$ be as above. Then
    \begin{align*}
        d_B( \Dgm(\mathrm{PH}_*([-1,1]^n,f)),~ \Dgm(\mathrm{PH}_*(\triang_m,f)) ) \leq L_f \cdot  2 \sqrt{n}/m.
    \end{align*}
\end{proposition}
Since the function $\log\CDpol{\mu_\samples}{d} : [-1,1]^n \to \R$ is differentiable and $[-1,1]^n$ is compact, $\log\CDpol{\mu_\samples}{d}$ is Lipschitz continuous, and so the above proposition applies to our setting. For a proof and a more detailed discussion on this approximation scheme, we refer to Appendix~\ref{APP: approximation}.
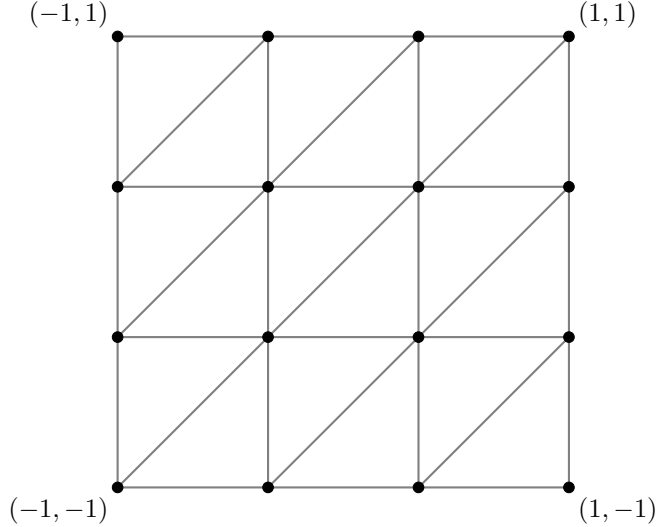
\begin{figure}
    \centering
    \begin{tikzpicture}
        \draw[gray, thick] (-3,3) -- (3,3);
        \draw[gray, thick] (-3,-1) -- (3,-1);
        \draw[gray, thick] (-3,1) -- (3,1);
        \draw[gray, thick] (-3,-3) -- (3,-3);
        \draw[gray, thick] (-3,3) -- (-3,-3);
        \draw[gray, thick] (-1,3) -- (-1,-3);
        \draw[gray, thick] (1,3) -- (1,-3);
        \draw[gray, thick] (3,3) -- (3,-3);
        \draw[gray, thick] (-3,1) -- (-1,3);
        \draw[gray, thick] (-3,-1) -- (1,3);
        \draw[gray, thick] (-3,-3) -- (3,3);
        \draw[gray, thick] (-1,-3) -- (3,1);
        \draw[gray, thick] (1,-3) -- (3,-1);

        \filldraw[black] (3,3) circle (2pt) node[anchor=south west]{$(1,1)$};
        \filldraw[black] (3,-3) circle (2pt) node[anchor=north west]{$(1,-1)$};
        \filldraw[black] (-3,3) circle (2pt) node[anchor=south east]{$(-1,1)$};
        \filldraw[black] (-3,-3) circle (2pt) node[anchor=north east]{$(-1,-1)$};
        \filldraw[black] (3,1) circle (2pt) node{};
        \filldraw[black] (3,-1) circle (2pt) node{};
        \filldraw[black] (1,1) circle (2pt) node{};
        \filldraw[black] (1,-1) circle (2pt) node{};
        \filldraw[black] (-1,1) circle (2pt) node{};
        \filldraw[black] (-1,3) circle (2pt) node{};
        \filldraw[black] (-1,-3) circle (2pt) node{};
        \filldraw[black] (1,3) circle (2pt) node{};
        \filldraw[black] (1,-3) circle (2pt) node{};
        \filldraw[black] (-1,-1) circle (2pt) node{};
        \filldraw[black] (-3,1) circle (2pt) node{};
        \filldraw[black] (-3,-1) circle (2pt) node{};
\end{tikzpicture}
    \caption{The Freudenthal triangulation of $[-1, 1]^2$ on the lattice $\frac{2}{m} \cdot \mathbb{Z}^2$ (with $m=3$).}%
    \label{FIG:Freudenthal}
\end{figure}
\section{Numerical examples}
\label{SEC:experiments}
In this section we illustrate our new scheme by computing the persistence diagram of $\persmod{d}{\samples}$ for various toy examples $\samples \subseteq [-1, 1]^n$ for $n=1, 2, 3$. In each case, the sets $\samples$ are drawn from a measure supported $[-1, 1]^n$, after which some additional noise may be added. We compute the Christoffel polynomials using the method outlined in Section~\ref{SEC:empirical} and the linear algebra packages of {\rm NumPy}~\cite{numpy} and {\rm Scipy}~\cite{scipy}. To \emph{approximate} the persistence of the resulting filtrations, we employ the method of Section~\ref{SEC:approximation}, where we set the resolution $m=250$ for $n=1,2$ and $m=50$ for $n=3$. See Appendix~\ref{APP: approximation}. The persistence of the resulting lower-star filtration is computed using {\rm Gudhi}~\cite{gudhi}. We also use Gudhi to compute the (exact) persistent homology of the \v{C}ech filtration.

We add noise to our data sets $\samples$ in two ways. 
We say we add \emph{uniform noise} when the noisy data $\tilde \samples$ is obtained from~$\samples$ by adding $M$ points chosen uniformly at random from $[-1, 1]^n$. We say we add \emph{Gaussian noise} (with standard deviation $\sigma \geq 0$), when $\tilde \samples$ is obtained from $\samples$ by adding to each coordinate of each sample $\x \in \samples$ an independently drawn perturbation $t \sim N(0, \sigma)$.

\subsection*{A univariate example.}
It is rather instructive to consider first a simple univariate example. Let $\samples \subseteq [-1, 1]$ be drawn from a uniform measure $\mu$ supported on five disjoint intervals $I_1, I_2, \dots, I_5 \subseteq [-1, 1]$. The corresponding Christoffel polynomials ($d=4, 8, 12$) and persistence diagrams for this situation are plotted in Figure~\ref{FIG:1Dexample}. We would expect the persistence diagram of $\persmod{d}{\samples}$ to reflect the simple topology of $\mathrm{supp(\mu)}$; namely we expect $\persmod{d}{\samples}$ to consist of five interval modules, each corresponding to one of the connected components of $\mathrm{supp}(\mu)$. For $d = 8, 12$, we observe that this is indeed the case. For $d=4$, however, we see there are only four interval modules.
In the one-dimensional setting, this can be explained rather nicely. Indeed, `birth-events' correspond to (local) \emph{minima} of $\CDpol{\mu_{\samples}}{d}$, and `death-events' correspond to (local) \emph{maxima}. %
% \footnote{This correspondence can be generalized to $n > 1$, however, it is significantly more complicated in those cases.}. 
As $\CDpol{\mu_{\samples}}{4}$ is of degree $8$, it can have at most $7$ critical points, resulting in four interval modules (one of which is of infinite length).

\subsection*{The figure eight.}
In Figure~\ref{FIG:figureeight}, we plot $\log \CDpol{\mu_{\samples}}{12}$ and $\persmod{12}{\samples}$ when $\samples$ is drawn from a measure supported on two circles in three different configurations: disjoint, just intersecting, and overlapping. We observe that $\persmod{12}{\samples}$ correctly captures the underlying topology in all three cases (for clarity, the number of overlapping points in the diagram is indicated). 

\subsection*{Feature sizing}
We consider the influence of feature size in the data on our persistence module. We draw samples $\samples$ from a configuration of two circles in $[-1, 1]^n$. In Figure~\ref{FIG:smallcircles}, we depict the diagram of $\persmod{10}{\samples}$ as we decrease the \emph{radius} of one of the circles. Similarly, in Figure~\ref{FIG:sparsecircles}, we depict the diagram of $\persmod{10}{\samples}$ as we decrease the \emph{number of samples} drawn from one of the circles. In both cases, the decrease in feature size corresponds to a decrease in the length of the corresponding interval in $\persmod{10}{\samples}$. Interestingly, this is due to an earlier time of death in the former case, and due to a later time of birth in the latter.

\begin{figure}[p]
    \centering
    \begin{subfigure}{0.56\textwidth}
    \centering
    \includegraphics[width=\textwidth, trim=1cm 0.5cm 1cm 1cm,clip]{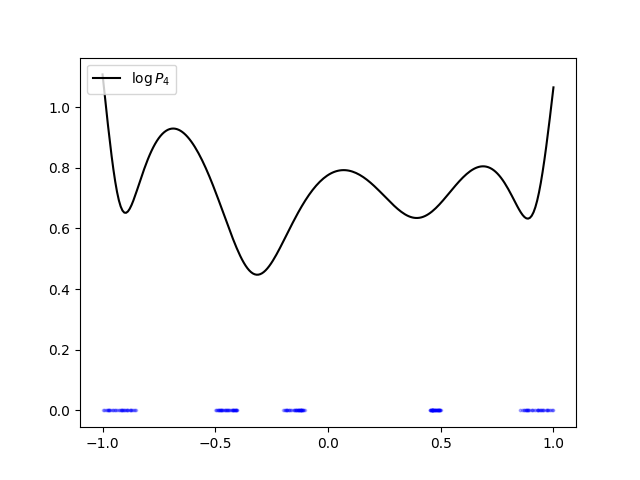}
        % \caption{$\log$-Christoffel polynomial (degree 4)}
    \end{subfigure}
    \begin{subfigure}{0.42\linewidth}
        \centering
         \includegraphics[width=\textwidth,trim=2.5cm 0.5cm 3cm 1cm,clip]{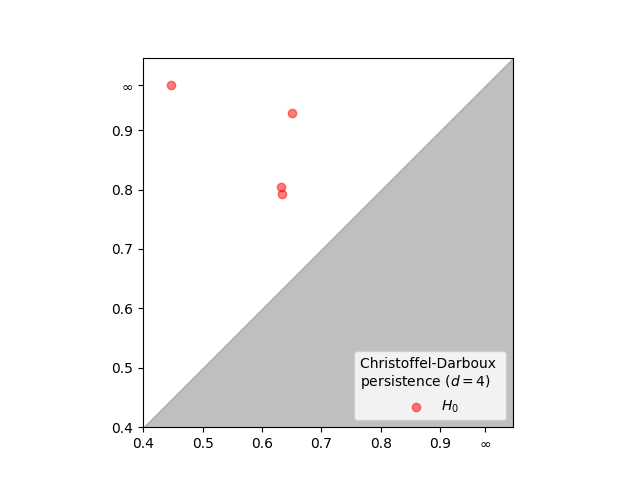}
        % \caption{Persistence diagram of $\persmod{4}{\samples}$}
    \end{subfigure}
    \begin{subfigure}{0.56\linewidth}
        \centering           \includegraphics[width=\textwidth, trim=1cm 0.5cm 1cm 1cm,clip]{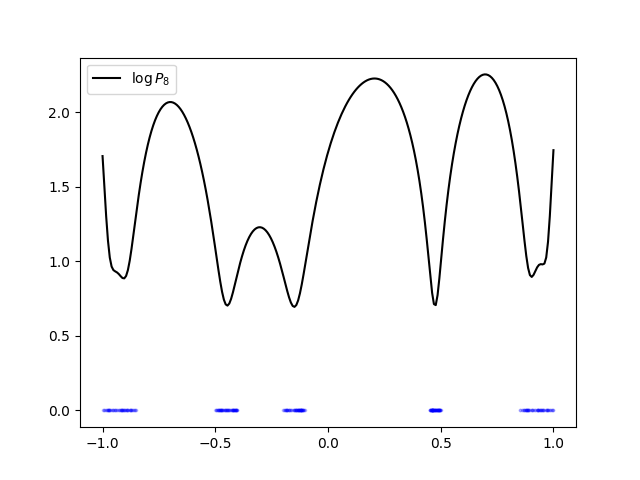}
        % \caption{$\log$-Christoffel polynomial (degree 8)}
    \end{subfigure}
        \begin{subfigure}{0.42\textwidth}
        \centering           \includegraphics[width=\textwidth,trim=2.5cm 0.5cm 3cm 1cm,clip]{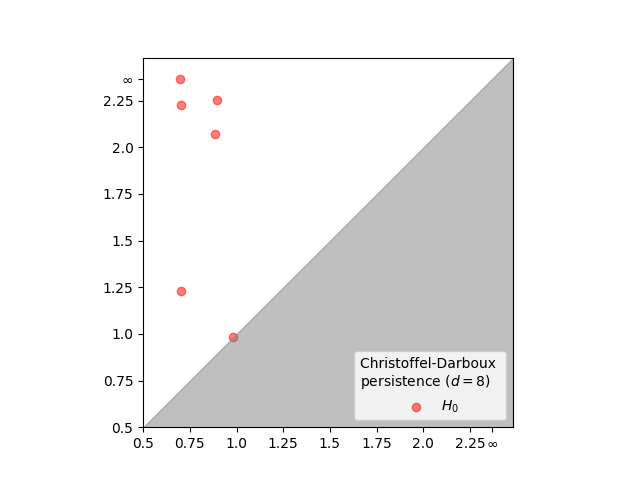}
        % \caption{Persistence diagram of $\persmod{8}{\samples}$}
    \end{subfigure}
    \begin{subfigure}{0.56\linewidth}
        \centering           \includegraphics[width=\textwidth, trim=1cm 0.5cm 1cm 1cm,clip]{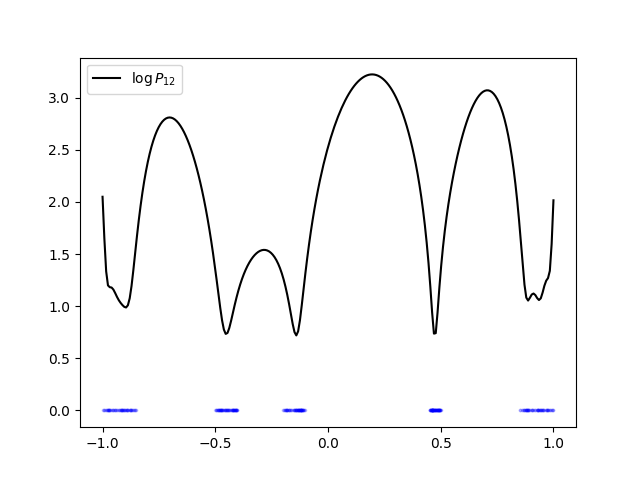}
        % \caption{$\log$-Christoffel polynomial (degree 12)}
    \end{subfigure}
    \begin{subfigure}{0.42\linewidth}
        \centering
         \includegraphics[width=\textwidth,trim=2.5cm 0.5cm 3cm 1cm,clip]{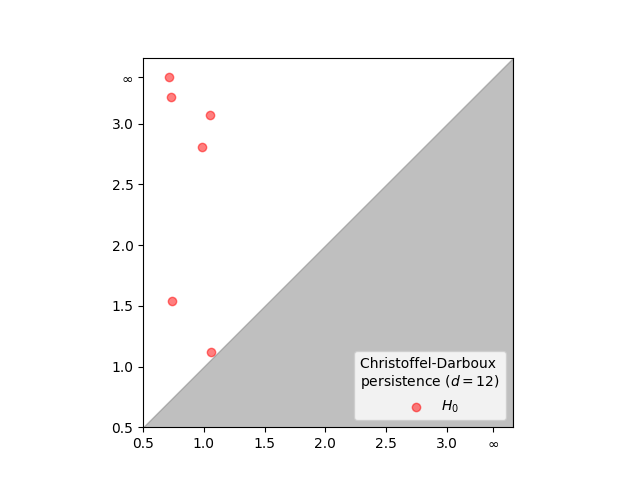}
        % \caption{Persistence diagram of $\persmod{12}{\samples}$}
    \end{subfigure}
    \caption{Left: Christoffel polynomials  ($d=4, 8 , 16$) for samples $\samples \subseteq [-1, 1]$ drawn from a measure supported on five intervals in $[-1, 1]$ (in blue, $N=500$). Right: diagrams of $\persmod{d}{\samples}$.}
    \label{FIG:1Dexample}
\end{figure}

\begin{figure}[p]
    \centering
    \begin{subfigure}{0.38\textwidth}
        \centering
    \includegraphics[width=\textwidth,trim=0.85cm 0.85cm 0.85cm 0.85cm,clip]{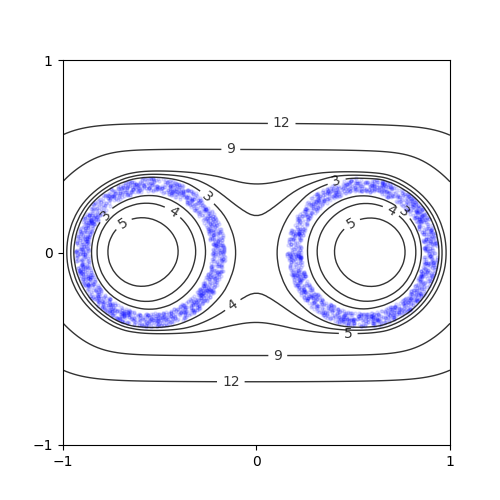}
        % \caption{Sample set $\samples_1$: Disjoint circles}
    \end{subfigure}
    \begin{subfigure}{0.38\linewidth}
        \centering
         \includegraphics[width=\textwidth,trim=2.5cm 0.5cm 3cm 1cm,clip]{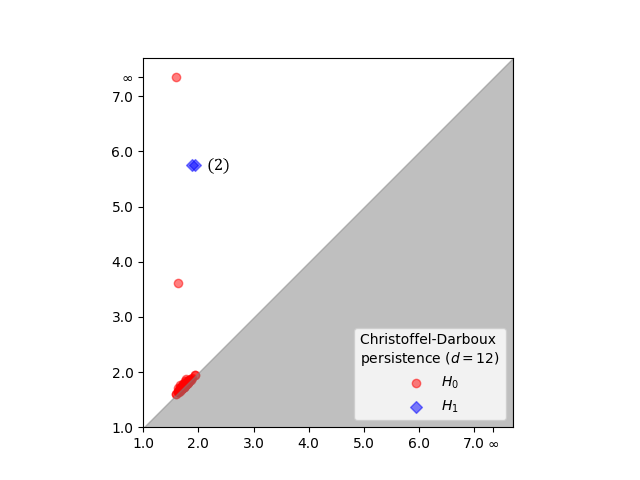}
        % \caption{Diagram of $\persmod{12}{\samples_1}$}
    \end{subfigure}
    \begin{subfigure}{0.38\linewidth}
        \centering \includegraphics[width=\textwidth,trim=0.85cm 0.85cm 0.85cm 0.85cm,clip]{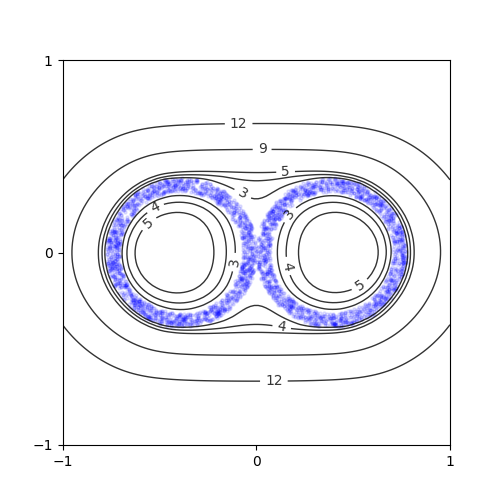}
        % \caption{Sample set $\samples_2$: Touching circles}
    \end{subfigure}
    \begin{subfigure}{0.38\textwidth}
        \centering           \includegraphics[width=\textwidth,trim=2.5cm 0.5cm 3cm 1cm,clip]{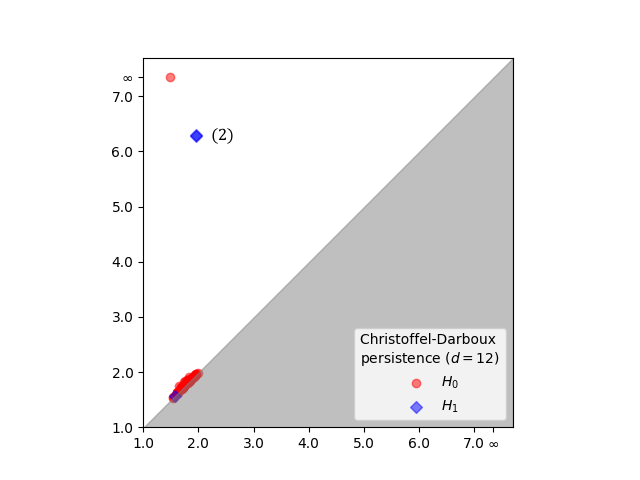}
        % \caption{Diagram of $\persmod{12}{\samples_2}$}
    \end{subfigure}
    \begin{subfigure}{0.38\linewidth}
        \centering           
        \includegraphics[width=\textwidth,trim=0.85cm 0.85cm 0.85cm 0.85cm,clip]{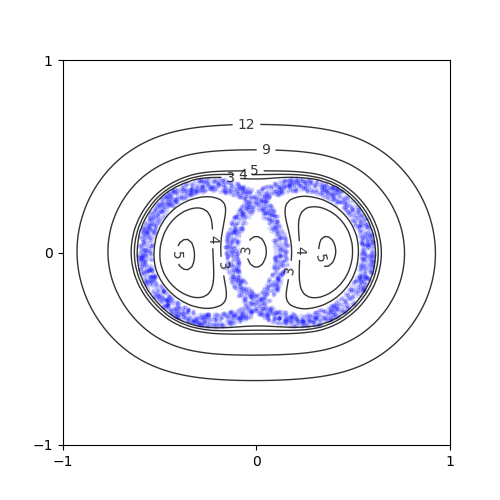}
        % \caption{Sample set $\samples_3$: Overlapping circles}
    \end{subfigure}
    \begin{subfigure}{0.38\linewidth}
        \centering
         \includegraphics[width=\textwidth,trim=2.5cm 0.5cm 3cm 1cm,clip]{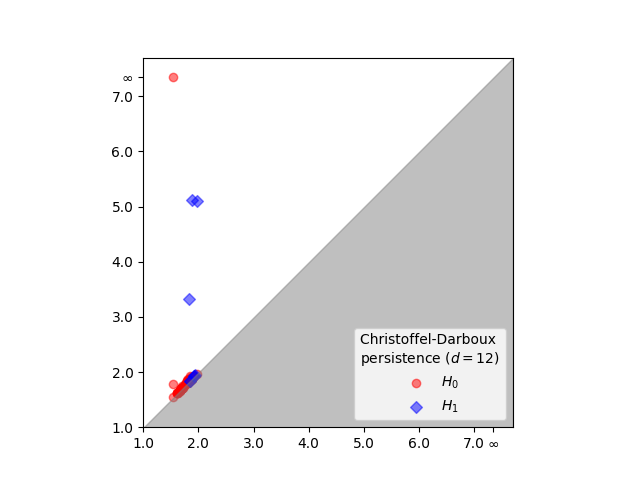}
        % \caption{Diagram of $\persmod{12}{\samples_3}$}
    \end{subfigure}
    \caption{Left: level sets (in black) of the Christoffel polynomial $\CDpol{\mu_{\samples}}{12}$ of degree $12$  for three sets of samples $\samples \subseteq [-1, 1]^n$ (in blue, $N = 3000$). Right: the corresponding diagrams of $\persmod{12}{\samples}$.}
    \label{FIG:figureeight}
\end{figure}
\begin{figure}[p]
    \centering
    \begin{subfigure}{0.25\textwidth}
    \centering
    \includegraphics[width=\textwidth,trim=0.85cm 0.85cm 0.85cm 0.85cm,clip]{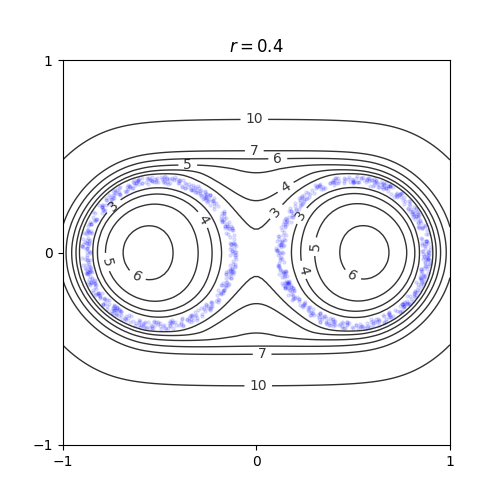}
    \end{subfigure}
    \begin{subfigure}{0.25\textwidth}
    \centering
    \includegraphics[width=\textwidth,trim=0.85cm 0.85cm 0.85cm 0.85cm,clip]{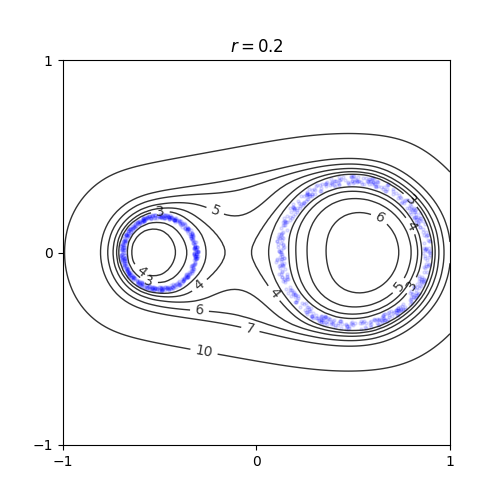}
    \end{subfigure}
    \begin{subfigure}{0.25\textwidth}
    \centering
    \includegraphics[width=\textwidth,trim=0.85cm 0.85cm 0.85cm 0.85cm,clip]{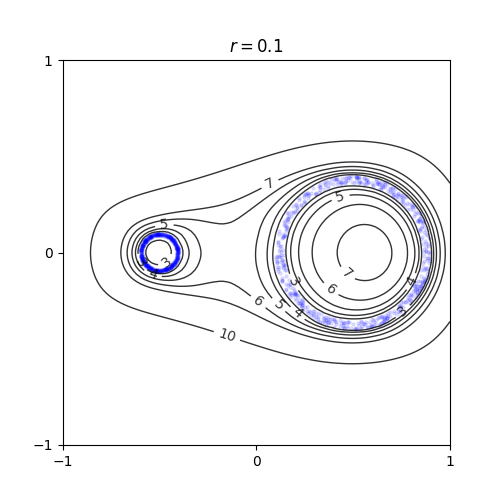}
    \end{subfigure}
    \begin{subfigure}{0.25\textwidth}
    \centering
    \includegraphics[width=\textwidth,trim=2.5cm 0.5cm 3cm 1cm,clip]{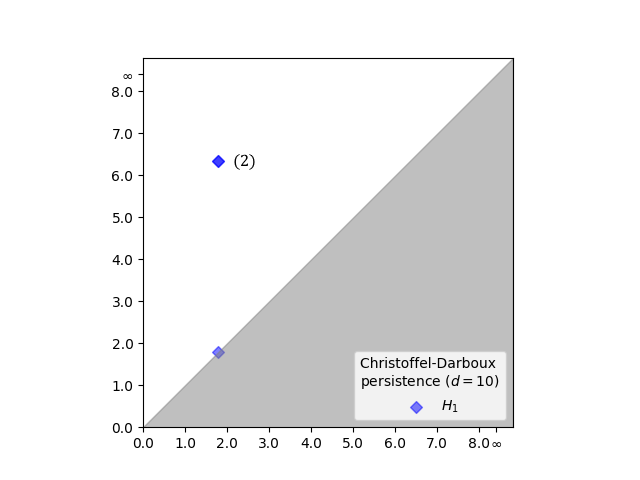}
    \end{subfigure}
        \begin{subfigure}{0.25\textwidth}
    \centering
    \includegraphics[width=\textwidth,trim=2.5cm 0.5cm 3cm 1cm,clip]{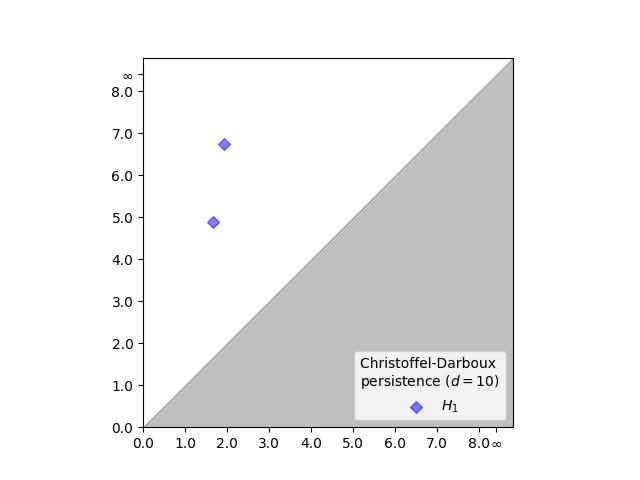}
    \end{subfigure}
        \begin{subfigure}{0.25\textwidth}
    \centering
    \includegraphics[width=\textwidth,trim=2.5cm 0.5cm 3cm 1cm,clip]{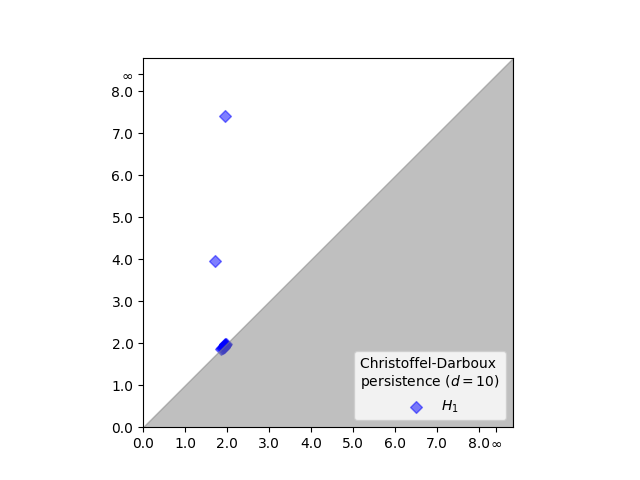}
    \end{subfigure}
    \caption{Level sets of $\log \CDpol{\mu_{\samples}}{10}$and diagrams of $\persmod{10}{\samples}$ for samples $\samples$ drawn from a measure supported on two circles. The radius $r$ of the left circle decreases. Only degree $1$ homology is shown.}
    \label{FIG:smallcircles}
\end{figure}
\begin{figure}[p]
    \centering
    \begin{subfigure}{0.25\textwidth}
    \centering
    \includegraphics[width=\textwidth,trim=0.85cm 0.85cm 0.85cm 0.85cm,clip]{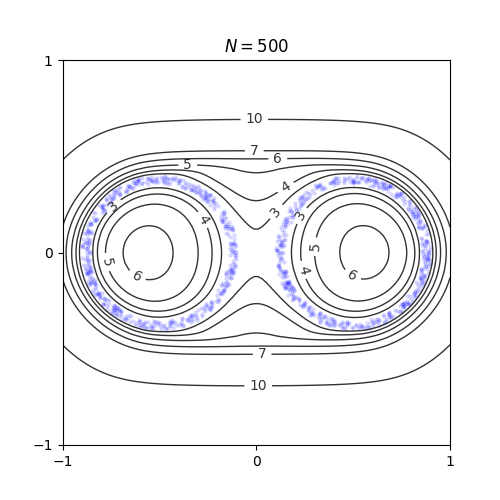}
    \end{subfigure}
    \begin{subfigure}{0.25\textwidth}
    \centering
    \includegraphics[width=\textwidth,trim=0.85cm 0.85cm 0.85cm 0.85cm,clip]{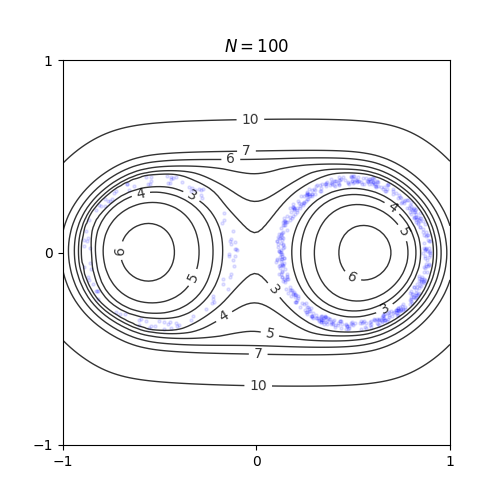}
    \end{subfigure}
    \begin{subfigure}{0.25\textwidth}
    \centering
    \includegraphics[width=\textwidth,trim=0.85cm 0.85cm 0.85cm 0.85cm,clip]{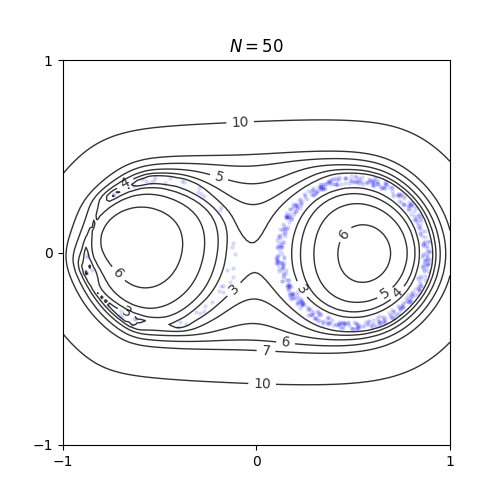}
    \end{subfigure}
    \begin{subfigure}{0.25\textwidth}
    \centering
    \includegraphics[width=\textwidth,trim=2.5cm 0.5cm 3cm 1cm,clip]{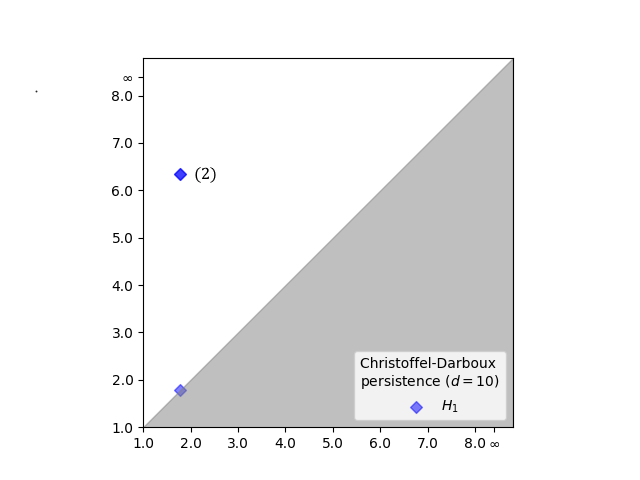}
    \end{subfigure}
        \begin{subfigure}{0.25\textwidth}
    \centering
    \includegraphics[width=\textwidth,trim=2.5cm 0.5cm 3cm 1cm,clip]{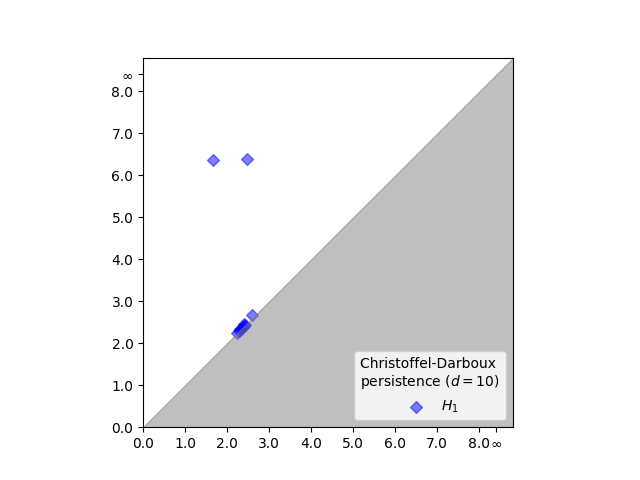}
    \end{subfigure}
        \begin{subfigure}{0.25\textwidth}
    \centering
    \includegraphics[width=\textwidth,trim=2.5cm 0.5cm 3cm 1cm,clip]{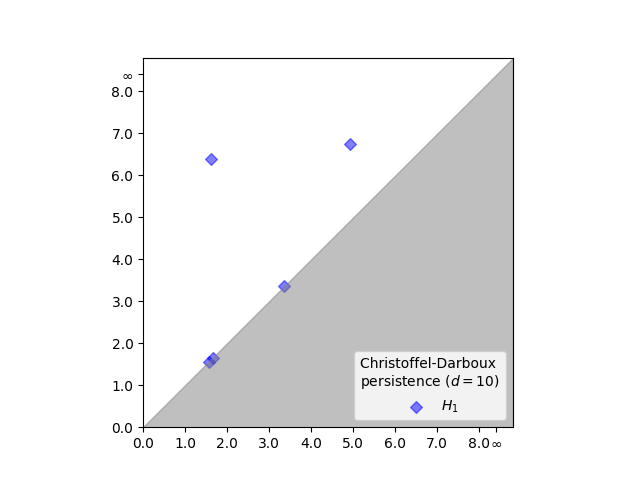}
    \end{subfigure}
    \caption{Level sets of $\log \CDpol{\mu_{\samples}}{10}$and diagrams of $\persmod{10}{\samples}$ for samples $\samples$ drawn from a measure supported on two circles. The number of samples $N$ drawn from the left circle decreases.
    Only degree $1$ homology is shown.}
    \label{FIG:sparsecircles}
\end{figure}

\subsection*{Comparison to the \v{C}ech module.}
We compare our new module $\persmod{12}{\samples}$ to \v{C}ech persistence in Figure~\ref{FIG:fourballs}, where $\samples$ is drawn from a measure supported on a ball, a triangle, and a square ($N=10000$). 
We consider four cases: a pure sample $\samples$, two samples with uniform noise ($M=50$ and $M=2500$) in $[-1,1]^n$, and a sample  with Gaussian noise ($\sigma = 0.03$). We observe that both the Christoffel-Darboux and \v{C}ech module are able to correctly capture the underlying topology for the pure sample and for the sample with Gaussian noise. However, only the  Christoffel-Darboux module is able to do so for the samples with uniform noise, whereas the \v{C}ech module recovers no meaningful information there.

\begin{figure}
    \centering
    \begin{subfigure}{0.32\textwidth}
    \centering
    \includegraphics[width=\textwidth,trim=0.85cm 0.85cm 0.85cm 0.85cm,clip]{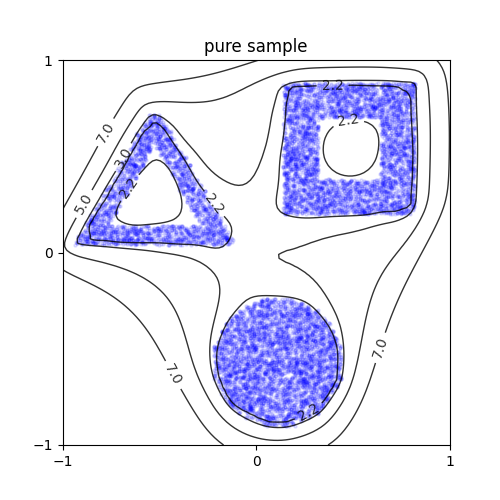}
    % \caption{Sample set $\samples$.}
    \end{subfigure}
    \begin{subfigure}{0.32\textwidth}
    \centering
    \includegraphics[width=\textwidth,trim=2.5cm 0.5cm 3cm 1cm,clip]{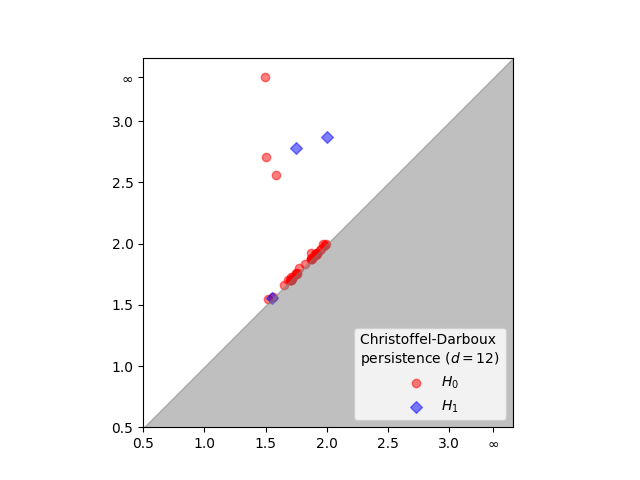}
    % \caption{Diagram of $\persmod{12}{\samples}$}
    \end{subfigure}
    \begin{subfigure}{0.32\textwidth}
    \centering
    \includegraphics[width=\textwidth,trim=2.5cm 0.5cm 3cm 1cm,clip]{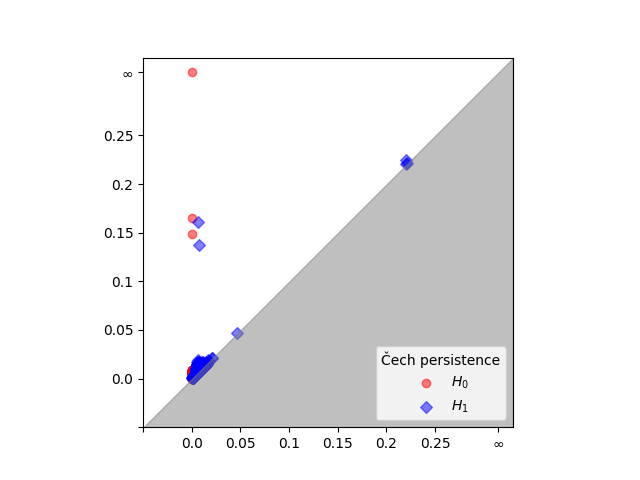}
    % \caption{Diagram of \v{C}ech persistence.}
    \end{subfigure}
    \centering
    \begin{subfigure}{0.32\textwidth}
    \centering
    \includegraphics[width=\textwidth,trim=0.85cm 0.85cm 0.85cm 0.85cm,clip]{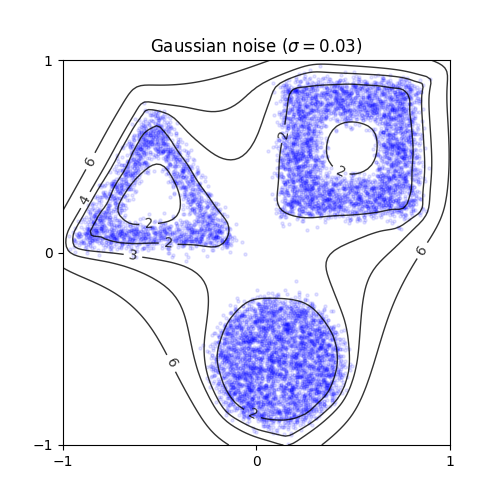}
    % \caption{Uniform noise ($M=50$)}
    \end{subfigure}
    \begin{subfigure}{0.32\textwidth}
    \centering
    \includegraphics[width=\textwidth,trim=2.5cm 0.5cm 3cm 1cm,clip]{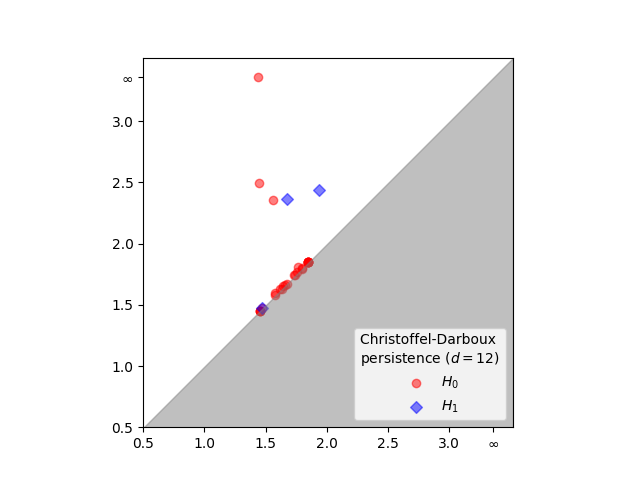}
    % \caption{Diagram of $\persmod{12}{\samples}$}
    \end{subfigure}
    \begin{subfigure}{0.32\textwidth}
    \centering
    \includegraphics[width=\textwidth,trim=2.5cm 0.5cm 3cm 1cm,clip]{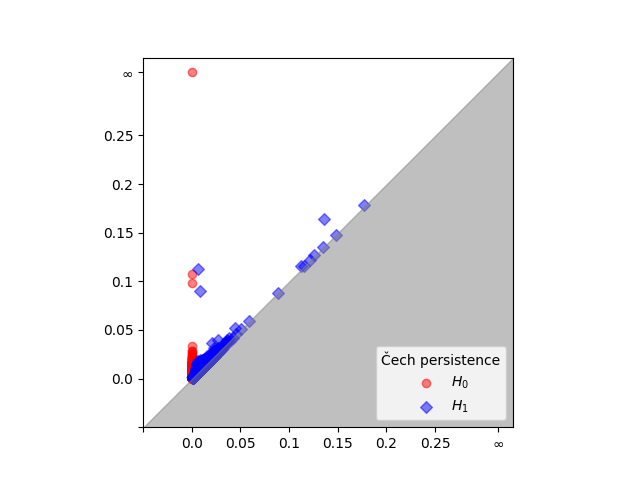}
    % \caption{Diagram of \v{C}ech persistence}
    \end{subfigure}
    \begin{subfigure}{0.32\textwidth}
    \centering
    \includegraphics[width=\textwidth,trim=0.85cm 0.85cm 0.85cm 0.85cm,clip]{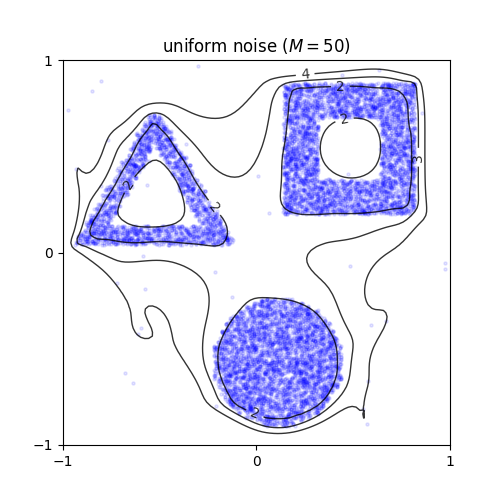}
    % \caption{Uniform noise ($M=2500$)}
    \end{subfigure}
    \begin{subfigure}{0.32\textwidth}
    \centering
    \includegraphics[width=\textwidth,trim=2.5cm 0.5cm 3cm 1cm,clip]{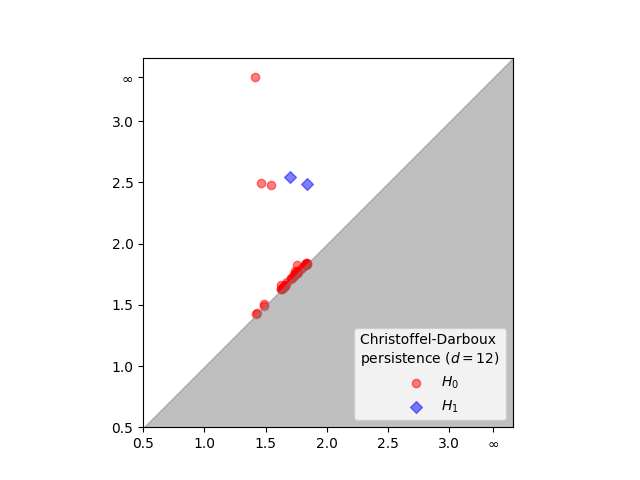}
    % \caption{Diagram of $\persmod{12}{\tilde \samples}$}
    \end{subfigure}
    \begin{subfigure}{0.32\textwidth}
    \centering
    \includegraphics[width=\textwidth,trim=2.5cm 0.5cm 3cm 1cm,clip]{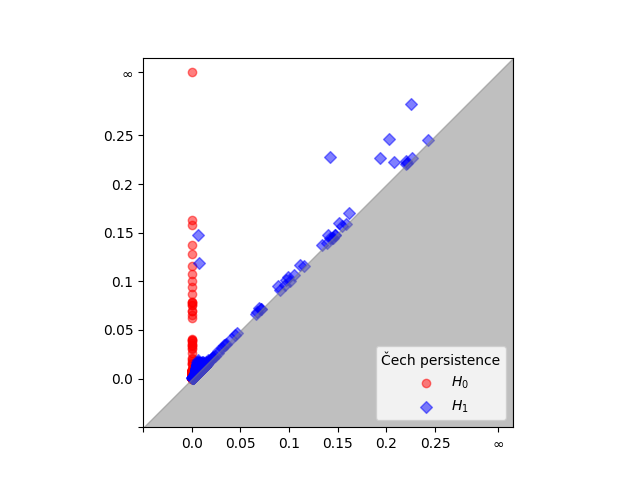}
    % \caption{Diagram of \v{C}ech persistence}
    \end{subfigure}
    \begin{subfigure}{0.32\textwidth}
    \centering
    \includegraphics[width=\textwidth,trim=0.85cm 0.85cm 0.85cm 0.85cm,clip]{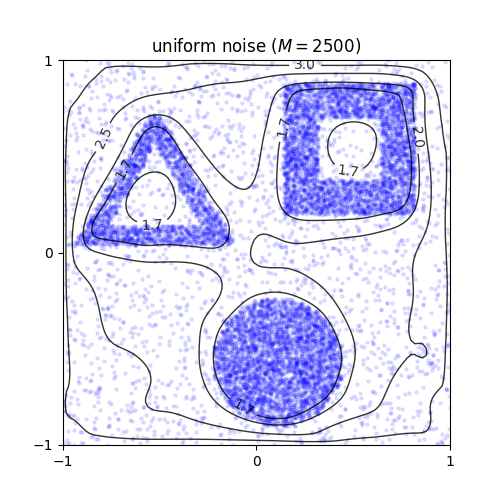}
    % \caption{Gaussian noise ($\sigma=0.03$)}
    \end{subfigure}
    \begin{subfigure}{0.32\textwidth}
    \centering
    \includegraphics[width=\textwidth,trim=2.5cm 0.5cm 3cm 1cm,clip]{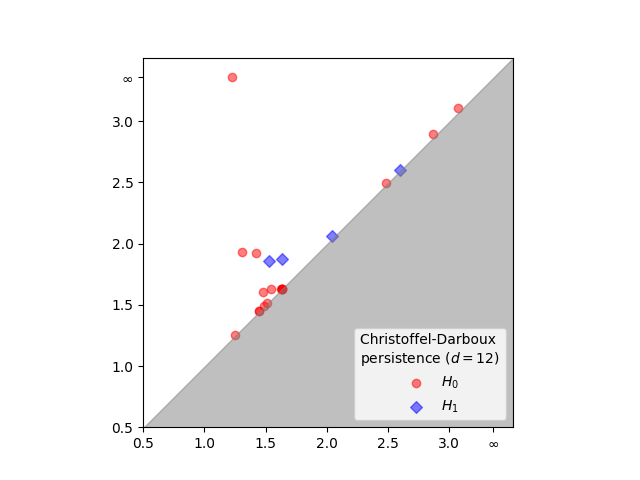}
    % \caption{Diagram of $\persmod{12}{\tilde \samples}$}
    \end{subfigure}
    \begin{subfigure}{0.32\textwidth}
    \centering
    \includegraphics[width=\textwidth,trim=2.5cm 0.5cm 3cm 1cm,clip]{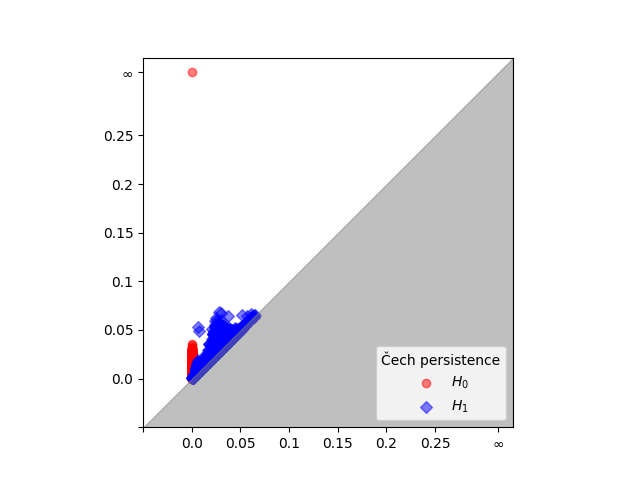}
    % \caption{Diagram of \v{C}ech persistence}
    \end{subfigure}
    \caption{Left: level sets of $\log \CDpol{\mu_{\samples}}{12}$ for sample sets $\samples$ in $[-1, 1]^2$ with different types of noise. Center: persistence diagrams of $\persmod{12}{\samples}$. Right: persistence diagrams for the \v{C}ech filtration.}
    \label{FIG:fourballs}
\end{figure}

\subsection*{Stability under uniform noise.}
We consider a set $\samples$ consisting of evenly spaced points on the $1$-skeleton of a cube\footnote{Because the cube skeleton is degenerate, we add a small amount of Gaussian noise ($\sigma = 0.025$) to $\samples$
to ensure the Christoffel polynomial for $\mu_{\samples}$ is well-defined.
} in $\R^3$ (50 points per edge) with edge-length $1.5$. 
The significant persistent features of~$\samples$ consist of a single interval (of infinite length) in $H_0$; five intervals in $H_1$ and one interval in $H_2$, see Figure~\ref{FIG:cubeskel}. We investigate how well the features in $H_1$ can be recovered after adding an increasing amount of uniform noise to $\samples$, comparing \v{C}ech persistence to $\persmod{d}{\samples}$, $d=6, 8, 10$. 
To measure this, we follow~\cite{Buchet2016} and use the \emph{signal-to-noise ratio}; meaning the ratio between the size of the smallest interval in $H_1$ inherent to~$\samples$ and the size of the largest interval (in $H_1$) induced by the noise. 
In Table~\ref{TAB:signal-to-noise}, we report the median signal-to-noise ratios over 100 experiments. We reiterate that the Christoffel-Darboux persistence is computed \emph{approximately}, which could affect these results. See Appendix~\ref{APP: approximation} for a more detailed discussion.
We observe that the ratios for the Christoffel-Darboux modules are much better than those for the \v{C}ech module. Furthermore, we note that the module of degree $d=6$ outperforms the modules of degree $d=8$ and $d=10$. This is consistent with our theoretical stability results in Section~\ref{SEC:stability}, which are stronger for small values of $d$.

\begin{table}[h]
    \rowcolors{2}{gray!25}{white}
    \centering
    \begin{tabular}{l|ccccccc}
    \rowcolor{white}
        \textbf{uniform noise} ($M$) & \textbf{baseline} & \textbf{25} & \textbf{50} & \textbf{100} & \textbf{250} & \textbf{500} & \textbf{1000} \\
      \hline
        \v{C}ech & $\gg 10$ & 3.6 & 2.4 & 2.0 & 1.6 & 1.4 & 1.3  \\
        Christoffel-Darboux ($d=6$) & $\gg 10$ & 7.8 & 5.7 & 5.3 & 3.7 & 2.3 & 1.2 \\
        Christoffel-Darboux ($d=8$) & $\gg 10$ & 4.9 & 2.8 & 2.6 & 2.5 & 2.1 & 1.2 \\
        Christoffel-Darboux ($d=10$) & $\gg 10$ & 8.8 & 4.0 & 2.3 & 1.9 & 1.8 & 1.3
    \end{tabular}
    \caption{Signal-to-noise ratios for persistent homology in dimension $1$ for the cube skeleton in the presence of uniform noise (median values over 100 experiments). See also Figure~\ref{FIG:cubeskel}.}
    \label{TAB:signal-to-noise}
\end{table}

\newpage
\begin{figure}
    \centering
    \begin{subfigure}{0.32\textwidth}
    \centering
    \includegraphics[width=\textwidth,trim=1.8cm 1.2cm 0.7cm 1cm,clip]{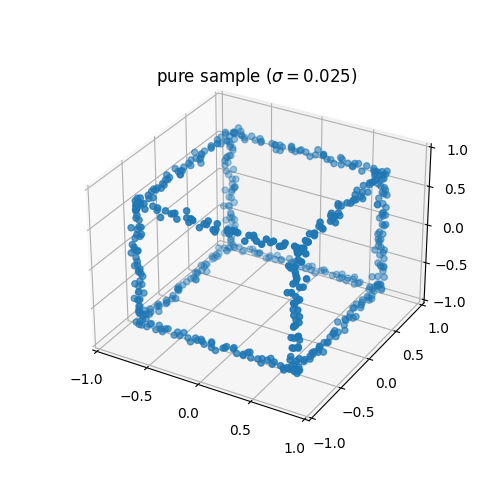}
        % \caption{Gaussian noise ($\sigma = 0.025$)}
    \end{subfigure}
    \begin{subfigure}{0.32\linewidth}
        \centering
         \includegraphics[width=\textwidth,trim=2.5cm 0.5cm 3cm 1cm,clip]{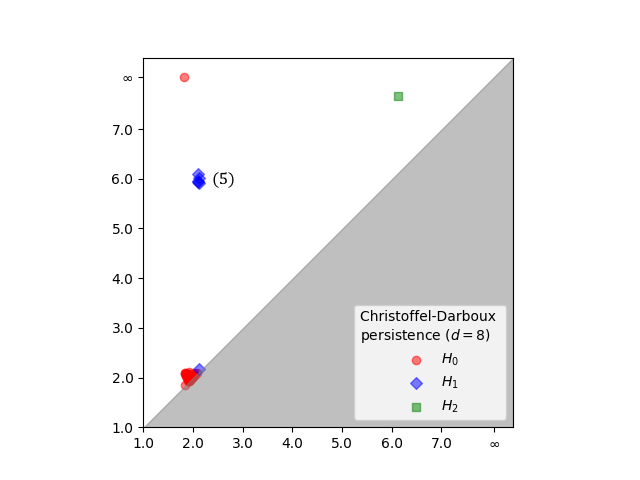}
        % \caption{Diagram of $\persmod{10}{\samples}$}
    \end{subfigure}
    \begin{subfigure}{0.32\linewidth}
        \centering
         \includegraphics[width=\textwidth,trim=2.5cm 0.5cm 3cm 1cm,clip]{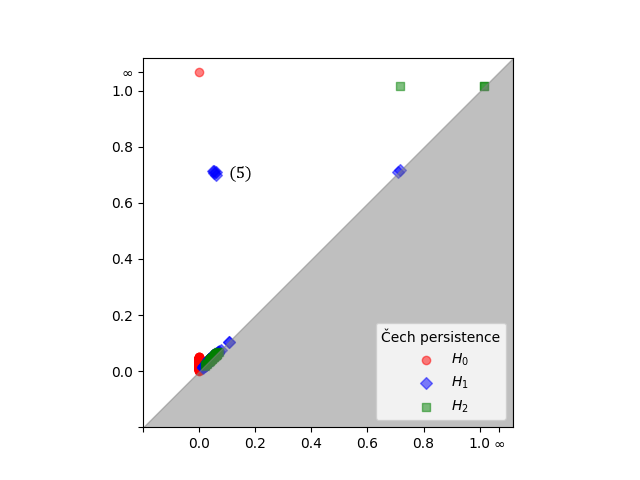}
        % \caption{Diagram of $\persmod{10}{\samples}$}
    \end{subfigure}
        \begin{subfigure}{0.32\textwidth}
    \centering
    \includegraphics[width=\textwidth,trim=1.8cm 1.2cm 0.7cm 1cm,clip]{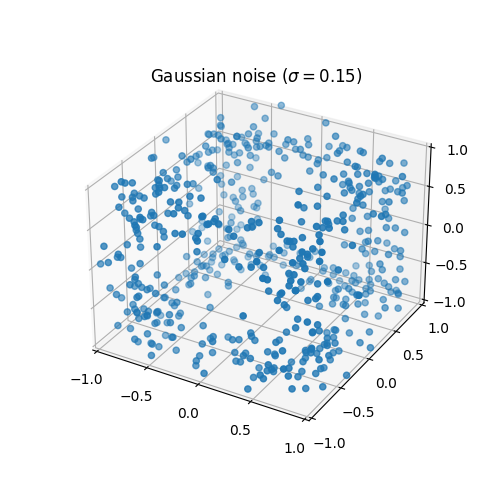}
        % \caption{Gaussian noise ($\sigma = 0.025$)}
    \end{subfigure}
    \begin{subfigure}{0.32\linewidth}
        \centering
         \includegraphics[width=\textwidth,trim=2.5cm 0.5cm 3cm 1cm,clip]{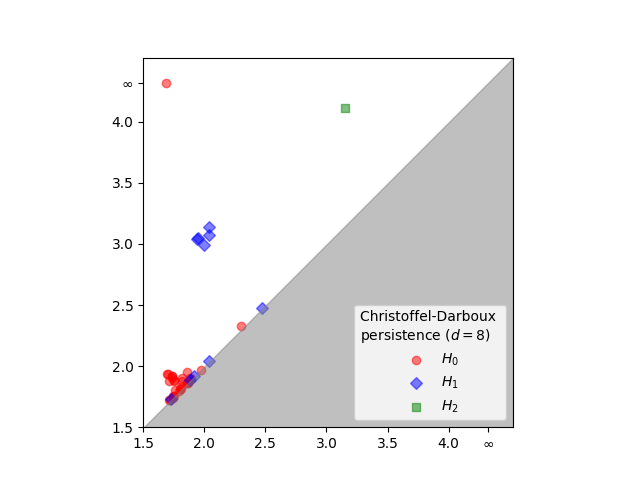}
        % \caption{Diagram of $\persmod{10}{\samples}$}
    \end{subfigure}
    \begin{subfigure}{0.32\linewidth}
        \centering
         \includegraphics[width=\textwidth,trim=2.5cm 0.5cm 3cm 1cm,clip]{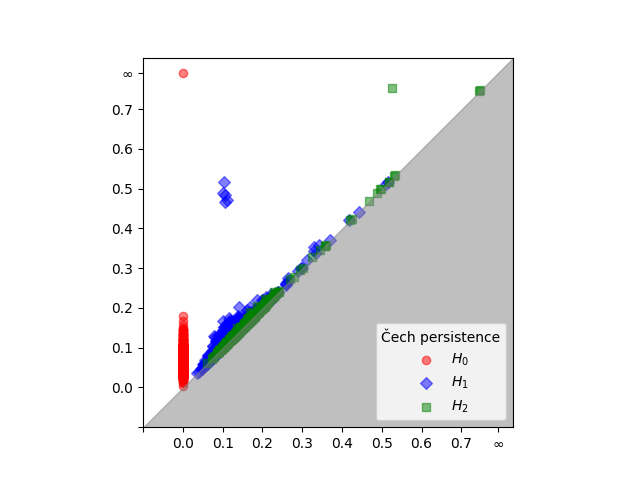}
        % \caption{Diagram of $\persmod{10}{\samples}$}
    \end{subfigure}
        \begin{subfigure}{0.32\textwidth}
    \centering
    \includegraphics[width=\textwidth,trim=1.8cm 1.2cm 0.7cm 1cm,clip]{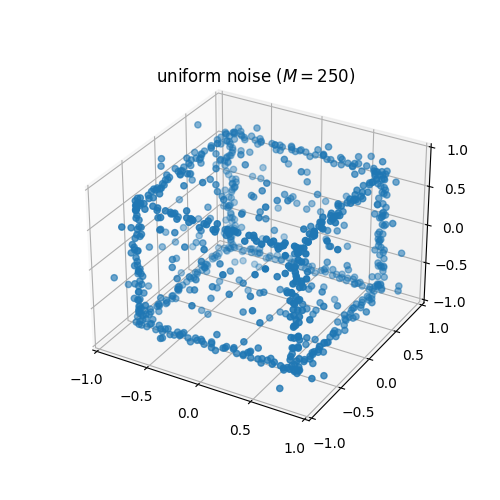}
        % \caption{Gaussian noise ($\sigma = 0.025$)}
    \end{subfigure}
    \begin{subfigure}{0.32\linewidth}
        \centering
         \includegraphics[width=\textwidth,trim=2.5cm 0.5cm 3cm 1cm,clip]{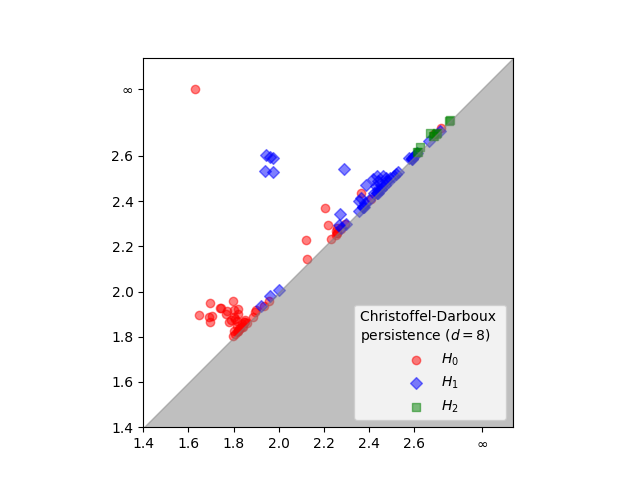}
        % \caption{Diagram of $\persmod{10}{\samples}$}
    \end{subfigure}
    \begin{subfigure}{0.32\linewidth}
        \centering
         \includegraphics[width=\textwidth,trim=2.5cm 0.5cm 3cm 1cm,clip]{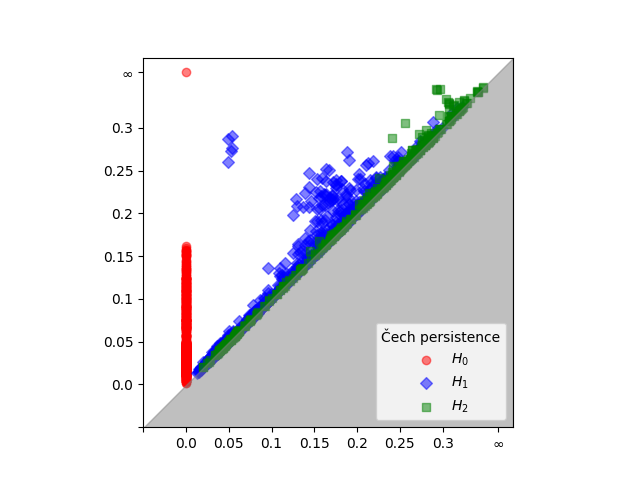}
        % \caption{Diagram of $\persmod{10}{\samples}$}
    \end{subfigure}
    \caption{
    Left: sample sets obtained from the $1$-skeleton of a cube in $[-1, 1]^3$ by adding Gaussian and uniform noise, respectively. Right: the corresponding \v{C}ech and Christoffel-Darboux persistence.}
    \label{FIG:cubeskel}
\end{figure}
\clearpage
\section{Discussion}
\label{SEC:discussion}
    We have introduced a new scheme for computing persistent homology of a point cloud in~$\R^n$, based on the theory of Christoffel-Darboux kernels. Our scheme is stable w.r.t. the Wasserstein distance. It admits an exact algorithm whose runtime is linear in the number of samples, but depends rather heavily on the ambient dimension $n$ and the degree $d$ of the kernel. In several examples ($n=1, 2, 3$), it was able to capture key topological features of the point cloud, even in the presence of uniform noise.
    
    \subsection*{Computing the persistent homology.}
    The persistence module $\persmod{d}{\samples}$ arises from a particularly simple filtration of a semialgebraic set by a polynomial. This is what allows us to invoke the result of Basu \& Karisani in Section~\ref{SEC:algorithm} to compute its persistence diagram. Their result in fact applies in a much more general setting, but the resulting algorithm is not very practical. Our present work motivates the search for \emph{effective} exact algorithms in simple special cases. Alternatively, it would be desirable to find theoretical guarantees that match the good practical performance of the approximation scheme described in Section~\ref{SEC:approximation}.
    
    \subsection*{Regularization.}
    Our stability results of Section~\ref{SEC:stability} depend on the `algebraic degeneracy' of the sample set $\samples$. Such dependence is undesirable, and not present in stability results for most conventional persistence modules. This dependence can potentially be avoided by considering a \emph{regularization} of the Christoffel polynomial, obtained by adding a small multiple of the identity to the moment matrix~\eqref{EQ:moment_matrix}:
        $M \gets M + \varepsilon \cdot \mathrm{Id}$.
    One can also think of this as adding a small multiple of the uniform measure on $[-1, 1]^n$ to the empirical measure~$\mu_\samples$, which ensures that the corresponding inner product is definite.
    % In particular, this allows one to work in the setting where $\samples$ is (near-)degenerate, i.e., where $M$ is (almost) singular. 
    In~\cite{Marxetal2021}, the authors already studied the 
    impact of such modifications in the setting of functional approximation. There, it allows them to work over (near-)degenerate sets $\samples$, while still accurately capturing the geometry of their problem.
    Preliminary experiments show this approach can be applied in our setting as well, and it would be very interesting to explore this further.
    
    \subsection*{Selecting the degree $d$.}
    Another important consideration is the selection of the degree $d$ of the Christoffel polynomial $\CDpol{\mu_{\samples}}{d}$. On the one hand, Proposition~\ref{EQ:christpol} and Theorem~\ref{THM:CDconvergence} suggest that the polynomial captures the support of the underlying measure $\mu$ better when $d$ is large. On the other hand, Proposition~\ref{PROP:logstab} and Table~\ref{TAB:signal-to-noise} suggest that $\persmod{d}{\samples}$ is more stable under perturbations of the data $\samples$ for smaller $d$.
    Furthermore, computing $\CDpol{\mu_{\samples}}{d}$ rapidly becomes more costly as $d$ grows.
    
\bibliographystyle{plainurl}
\bibliography{CDbib}
    
    \newpage
    \appendix
    \section{Proofs of technical statements}
    \label{APP:technical}
\begin{lemma}[Generalized Markov Brothers' inequality~\cite{Skylaga}, see also Theorem 1 in~\cite{HarrisMarkov}] \label{LEM:Markov}
Let $f \in \R[\x]$ be a polynomial of degree~$d$. Then for all $\x, \y \in [-1, 1]^n$, we have:
\[
    |f(\x) - f(\y)| \leq d^2 \cdot  \|f\|_\infty \cdot \| \x - \y \|_2.
\]
That is, $f$ is Lipschitz continuous on $[-1,1]^n$ with constant $L_f = d^2 \cdot \| f\|_\infty$.
\end{lemma}
\begin{lemma}[Technical property of Wasserstein distance] \label{LEM:cWasserstein}
Let $\eta > 0$ and let $c : \R^n \rightarrow \R$ be any function satisfying:
\[
    |c(\x) - c(\y)| \leq \eta \cdot \|\x - \y\|_2 \quad (\x \in \samples,~\y \in \sampless).
\]
Then we have:
\[
    |\sum_{\x \in \samples} c(\x) \mu_\samples(\x) - \sum_{\y \in \sampless} c(\y) \mu_{\sampless}(\y)| \leq \eta \cdot d_W(\mu_\samples, \mu_\sampless).
\]
\end{lemma}

\begin{proof}
Let $\gamma$ be the optimum solution to the program~\eqref{EQ:Wasserstein} defining the Wasserstein distance $d_W(\mu_\samples, \mu_\sampless)$. Then:
\begin{align*}
    |\sum_{\x \in \samples} c(\x) \mu_\samples(\x) - \sum_{\y \in \sampless} c(\y) d\mu_{\sampless}(\y)| 
    &= |\sum_{\x \in \samples} c(\x) \sum_{\y \in \sampless} \gamma(\x, \y) - \sum_{\y \in \sampless} c(\y) \sum_{\x \in \samples} \gamma(\x, \y)| \\
    &\leq \sum_{\x \in \samples, \y \in \sampless} |c(\x) - c(\y)| \gamma(\x, \y) \\
    &\leq \sum_{\x \in \samples, \y \in \sampless} \eta \cdot \|\x-\y\|_2 \cdot \gamma(\x, \y) \\
    &= \eta \cdot d_W(\samples, \sampless). \qedhere
\end{align*}
\end{proof}

\begin{proof}[Proof of Theorem~\ref{THM:wasstab}]
Choose an orthonormal basis $\basis = (b_\alpha)_{|\alpha| \leq d}$ for $\R[\x]_d$ w.r.t.~$\langle\cdot, \cdot \rangle_{\mu_\samples}$. We then have:
\begin{align}
    \CDpol{\mu_\samples}{d}(\x) &= \basis(\x)^\top \basis(\x),\\
    \CDpol{\mu_\sampless}{d}(\x) &= \basis(\x)^\top M^{-1} \basis(\x),
\end{align}
where $M := M^{\mu_{\sampless}}_d(\basis)$ is the moment matrix~\eqref{EQ:moment_matrix}. Thus, for any $\x \in \R^n$, we have:
\begin{equation}
    |\CDpol{\mu_{\samples}}{d}(\x) - \CDpol{\mu_{\sampless}}{d}(\x)|
    = |\basis(\x)^\top (\mathrm{Id} - M^{-1}) \basis(\x)|.
\end{equation}
Now, since $M \succ 0$ is positive definite, it has an invertible matrix square root $M^{\frac{1}{2}}$, satisfying $M = \big(M^{\frac{1}{2}}\big)^\top M^{\frac{1}{2}}$. We can therefore write:
\[
    \basis(\x)^{\top} \basis(\x) = \big(M^{-\frac{1}{2}}\basis(\x) \big)^\top M \big (M^{-\frac{1}{2}}\basis(\x) \big).
\]
We then get:
\begin{align*}
|\basis(\x)^\top (\mathrm{Id} - M^{-1}) \basis(\x)| 
&= |(M^{-\frac{1}{2}}\basis(\x))^\top (M - \mathrm{Id}) (M^{-\frac{1}{2}}\basis(\x))| \\
&\leq \|M - \mathrm{Id}\|_{\rm op} \cdot \basis(\x)^\top M^{-1} \basis(\x) = \|M - \mathrm{Id}\|_{\rm op} \cdot \CDpol{\mu_{\sampless}}{d}(\x).
\end{align*}
We thus want to bound $\|M - \mathrm{Id}\|_{\rm op}$. For this, note that the entries of $M - \mathrm{Id}$ are given by:
\[
    (M - \mathrm{Id})_{\alpha, \beta} = \sum_{\y \in \sampless} b_\alpha(\y) b_\beta(\y) \mu_\sampless(\y) - \sum_{\x \in \samples} b_\alpha(\x) b_\beta(\x) \mu_\samples(\x).
\]
Now, taking $c(\x) = b_\alpha(\x)b_\beta(\x)$ in Lemma~\ref{LEM:cWasserstein}, we find that:
\[
    |\sum_{\y \in \sampless} b_\alpha(\y) b_\beta(\y) \mu_\sampless(\y) - \sum_{\x \in \samples} b_\alpha(\x) b_\beta(\x) \mu_\samples(\x)| \leq L \cdot d_W(\mu_\samples, \mu_\sampless),
\]
whenever $L > 0$ is a (uniform) upper bound on the Lipschitz constants of the polynomials $b_{\alpha\beta}(\x) := b_\alpha(\x) b_\beta(\x)$, $|\alpha|, |\beta| \leq d$. As the operator norm of a matrix can be bounded by its maximum absolute row sum, it follows immediately that 
\[
    \|M-\mathrm{Id}\|_{\rm op} \leq s(n, d) \cdot L \cdot d_W(\mu_\samples, \mu_\sampless).
\]
It remains to find such a bound $L$. Note that each $b_{\alpha\beta}$ is a polynomial of degree~$2d$. Note further that:
\[
    \|b_{\alpha \beta}\|_{\infty} \leq \|b_{\alpha}\|_{\infty} \cdot \|b_{\beta}\|_{\infty} \leq \max_{|\gamma| \leq d} \|b_\gamma\|_{\infty}^2 = \max_{|\gamma| \leq d} \|b_\gamma^2\|_{\infty} \leq \| \CDpol{\mu_{\samples}}{d}\|_\infty.
\]
Here, we use the fact that $\CDpol{\mu_{\samples}}{d}(\x) = \sum_{|\alpha| \leq d} b_\alpha(\x)^2$.
In light of Lemma~\ref{LEM:Markov}, we may thus choose ${L = 4d^2 \cdot \|\CDpol{\mu_{\samples}}{d}\|_\infty}$. Putting things together, we conclude that:
\[
|\CDpol{\mu_{\samples}}{d}(\x) - \CDpol{\mu_{\sampless}}{d}(\x)| \leq 4 \cdot s(n, d) \cdot d^2 \cdot \|\CDpol{\mu_\samples}{d}\|_{\infty} \cdot d_W(\mu_\samples, \mu_{\sampless}) \cdot \CDpol{\mu_\sampless}{d}(\x). \qedhere
\]
\end{proof}

\begin{proof}[Proof of Proposition~\ref{PROP:logstab}]
    Let $\x \in \X$ and suppose that $\CDpol{\mu_\samples}{d}(\x) \geq \CDpol{\mu_\sampless}{d}(\x)$. By Theorem~\ref{THM:wasstab} we know that 
    \begin{align*}
        0 \leq \CDpol{\mu_{\samples}}{d}(\x) - \CDpol{\mu_{\sampless}}{d}(\x) = |\CDpol{\mu_{\samples}}{d}(\x) - \CDpol{\mu_{\sampless}}{d}(\x)| \leq C_{n, d} \cdot \|\CDpol{\mu_\samples}{d}\|_{\infty} \cdot d_W(\mu_\samples, \mu_{\sampless}) \cdot \CDpol{\mu_\sampless}{d}(\x).
    \end{align*}
    Dividing by $\CDpol{\mu_{\sampless}}{d}(\x)$ gives us 
    \begin{align*}
        0 \leq \CDpol{\mu_{\samples}}{d}(\x)/\CDpol{\mu_{\sampless}}{d}(\x) - 1 &\leq C_{n, d} \cdot \|\CDpol{\mu_\samples}{d}\|_{\infty} \cdot d_W(\mu_\samples, \mu_{\sampless}) \\
        &\leq C_{n, d} \cdot \max\{\|\CDpol{\mu_\samples}{d}\|_{\infty},\|\CDpol{\mu_\sampless}{d}\|_{\infty}\} \cdot d_W(\mu_\samples, \mu_{\sampless}) 
    \end{align*}
    and therefore we see that 
    \begin{align*}
        0 = \log(1) &\leq \log(\CDpol{\mu_\samples}{d}(\x)) - \log(\CDpol{\mu_\sampless}{d}(\x)) \\
        &\leq \log \big( C_{n, d} \cdot \max\{\|\CDpol{\mu_\samples}{d}\|_{\infty},\|\CDpol{\mu_\sampless}{d}\|_{\infty}\} \cdot d_W(\mu_\samples, \mu_{\sampless}) +1 \big),
    \end{align*}
    and so $|\log(\CDpol{\mu_\samples}{d}(\x)) - \log(\CDpol{\mu_\sampless}{d}(\x)) |\leq \log(C_{n, d} \cdot \max\{\|\CDpol{\mu_\samples}{d}\|_{\infty},\|\CDpol{\mu_\sampless}{d}\|_{\infty}\} \cdot d_W(\mu_\samples, \mu_{\sampless}) +1)$. By an identical argument, in the case that $\CDpol{\mu_\samples}{d}(\x) \leq \CDpol{\mu_\sampless}{d}(\x)$, we also get that 
    \begin{align*}
        |\log(\CDpol{\mu_\samples}{d}(\x)) - \log(\CDpol{\mu_\sampless}{d}(\x)) | &= |\log(\CDpol{\mu_\sampless}{d}(\x)) - \log(\CDpol{\mu_\samples}{d}(\x)) | \\ &\leq \log(C_{n, d} \cdot \max\{\|\CDpol{\mu_\samples}{d}\|_{\infty},\|\CDpol{\mu_\sampless}{d}\|_{\infty}\} \cdot d_W(\mu_\samples, \mu_{\sampless}) +1)
    \end{align*}
    and so this inequality holds for all $\x \in \X$. Maximizing over all $\x \in \X$ then gives the desired statement. 
\end{proof}

\section{Details on the approximation scheme}
\label{APP: approximation}
In this section we elaborate on the approximation scheme outlined in Section~\ref{SEC:approximation}, and prove Proposition~\ref{PROP:Lipschitzstability}. Throughout, $f: [-1,1]^n \to \R$ denotes a Lipschitz continuous function with Lipschitz constant $L_f$.

Recall that a $k$-dimensional \emph{simplex} $\sigma$ in $\R^n$ is the convex hull of $k+1$ affinely independent vectors, and that the convex hull of a subset of those vectors is called a \emph{face} of $\sigma$. For $\sigma = \mathrm{conv}(v_1,\ldots,v_{k+1})$, we denote $V(\sigma) = \{v_1,\ldots,v_m\}$. A \emph{triangulation} of $[-1,1]^n$ is a collection of simplices $\triang$ which satisfy the following:
\begin{itemize}
    \item If $\sigma \in \triang$, and $\rho$ is a face of $\sigma$, then $\rho \in \triang$.
    \item For any $\sigma, \rho \in \triang$, the intersection $\sigma \cap \rho$ is a face of both $\sigma$ and $\rho$.
    \item $|\triang | := \bigcup_{\sigma \in \triang} \sigma = [-1,1]^n$
\end{itemize}
% We say that $\triang$ has diameter:
% \[
% \mathrm{diam}(\triang) := \max_{\sigma \in \triang} \mathrm{diam}(\sigma).
% \]
The first two conditions above say that $\triang$ is a \emph{simplicial complex}, and the last one ensures that $\triang$
captures the topology of $[-1,1]^n$.  We consider the following filtration on $\triang$:
For each $t \in \R$, we define
\begin{align*}
    \triang_t :=\{\sigma  \in \triang \, | \, f(v_i) \leq t \quad \forall i=1,\ldots,m\}.
\end{align*}

Clearly, if $s \leq t$, then $\triang_s \subseteq \triang_t$. The collection $\{\triang_t\}_{t\in \R}$ is called the \emph{lower-star filtration} of $\triang$ with respect to $f$. Applying \emph{simplicial} homology to this filtration defines a persistence module, denoted by $\mathrm{PH}_*(\triang,f)$, which can be computed \cite{persistencecubic,persistentcomplexity} in polynomial time. This persistence is isomorphic to the sublevel set persistent homology of a particular piece-wise linear function defined on $[-1,1]^n$, which we will describe now. If $\triang$ is a triangulation of $[-1,1^n]$, then any $\x \in [-1,1]^n$ is contained in a unique simplex $\sigma \in \triang$ of minimal dimension. Therefore, we may write $\x$ as a convex combination of vertices $v_i$ of $\sigma$:
\begin{equation} \label{EQ:convcomb}
    \x = \sum_{i=1}^k \alpha_i v_i.
\end{equation}
The piece-wise linear approximation of $f$ (with respect to $\triang$) is now defined as:
\[
    \approxf(\x) = \sum_{i=1}^k \alpha_i f(v_i),
\]
where the $v_i$ and $\alpha_i$ are as in~\eqref{EQ:convcomb}. We can relate $f$, with Lipschitz constant $L_f$, to $\overline{f}$:

\begin{proposition} \label{PROP:approxerror}
    Let $f, \approxf : [-1,1]^n \to \R$ be as in the above. Then we have:
    \[
        \|f - \approxf\|_{\infty} \leq L_f \cdot \mathrm{diam}(\triang), \text{ where }     \mathrm{diam}(\triang) := \max_{\sigma \in \triang} \mathrm{diam}(\sigma).
    \]
    Using the fact that $\mathrm{PH}_*([-1,1]^n,\overline{f})) \cong \mathrm{PH}_*(\triang,f))$ \cite{bookEdelsbrunner}, we obtain as a  direct consequence of Corollary~\ref{COR:contstab} that
    \[
        d_B(\Dgm(\mathrm{PH}_*([-1,1]^n,f)),~ \Dgm(\mathrm{PH}_*(\triang,f))) \leq L_f \cdot \mathrm{diam}(\triang).
    \]
\end{proposition}

\begin{proof}
    For any $\x \in \X$, we write $\x = \sum_{i=1}^k \alpha_i v_i$ as in \ref{EQ:convcomb}. We then see that:
    \begin{align*}
        |f(\x) - \approxf(\x)| &= |f(\x) - \sum_{i =1} ^k \alpha_i f(v_i)|
        = | \sum_{i =1} ^k \alpha_i \cdot (f(\x) - f(v_i))| \\
        &\leq  \sum_{i =1} ^k \alpha_i \cdot  |f(\x) - f(v_i)| 
        \leq \sum_{i =1} ^k \alpha_i \cdot L_f \cdot \|\x - v_i\|_2 \\
        &\leq \sum_{i =1} ^k \alpha_i \cdot L_f \cdot \mathrm{diam}(\triang) = L_f \cdot \mathrm{diam}(\triang). \qedhere
    \end{align*}
\end{proof}

We now explain how we triangulate the hypercube. For each $m \in \N$, we construct a triangulation $\triang_m$ of $[-1,1]^n$. The cube contains a regular grid $\Gamma_m$ of $(m+1)^n$ points, consisting of the lattice points of $\frac{2}{m} \cdot \mathbb{Z}$ contained in $[-1,1]^n$. We define $\triang_m$ to be the Freudenthal triangulation \cite{Freudenthal,eaves1984course} of $[-1,1]^n$ using the grid $\Gamma_m$ as vertex set. The Freudenthal triangulation of $[-1,1]^n $ triangulates each of the hypercubes of width $2/m$ defined by the grid $\Gamma_m$. See Figure~\ref{FIG:Freudenthal} for the case $n=2$, $m=3$. In each of these small hypercubes, the diagonal is part of the triangulation, so it follows that $\mathrm{diam}(\triang_m) = 2 \sqrt{n}/m$.  Together with Proposition~\ref{PROP:approxerror} this gives us: 

\begin{proposition}[Restatement of Prop.~\ref{PROP:Lipschitzstability}]
    Let $f: [-1,1]^n \to \R$ be a Lipschitz continuous function with Lipschitz constant $L_f$, choose $m\in \N$, and let $\triang_m$ be as above. Then
    \begin{align*}
        d_B( \Dgm(\mathrm{PH}_*([-1,1]^n,f)),~ \Dgm(\mathrm{PH}_*(\triang_m,f)) ) \leq L_f \cdot  2 \sqrt{n}/m.
    \end{align*}
\end{proposition}

Since the function $\log\CDpol{\mu_\samples}{d} : [-1,1]^n \to \R$ is differentiable and $[-1,1]^n$ is compact, $\log\CDpol{\mu_\samples}{d}$ is Lipschitz continuous, and so the above proposition applies to our setting. The proposition tells us that the bottleneck distance between the persistence diagram we obtain, and $\Dgm(\persmod{d}{\samples})$, is at most $L_{\log\CDpol{\mu_\samples}{d}} \cdot 2 \sqrt{n}/m$. We also remark that the construction of $\triang_m$, and the computation of the persistent homology of the lower-star filtration, can be done in polynomial time.

\subsection{Numerical validation}
Proposition~\ref{PROP:Lipschitzstability} tells us that our approximation converges to the real Christoffel-Darboux persistence as $m$ increases. However, we do not have explicit access to the Lipschitz constant of $\CDpol{\mu_{\samples}}{d}$, and in practice, this number may be rather large. It is possible to give a slightly sharper bound than the one presented in Proposition~\ref{PROP:Lipschitzstability}, but this still does not give a satisfactory guarantee on the quality of our approximation.
Nonetheless, the approximation turns out to converge quickly as the resolution $m$ increases. We illustrate this with two examples. 

In Figure~\ref{FIG:ThreeShapesResIncrease}, we show the change in Bottleneck distance between subsequent diagrams $\mathrm{PH}_*(\triang_m, ~\log \CDpol{\mu_{\samples}}{12})$ as we increase the resolution $m$ in increments of $10$. Here, $\samples \subseteq [-1, 1]^n$ is drawn as in Figure~\ref{FIG:fourballs}. We observe that the distance between subsequent diagrams tends to zero quickly, both for a pure sample, and for a sample with uniform noise.

In Figure~\ref{FIG:CubeResIncrease}, we show the change in observed signal-to-noise ratios in $\persmod{d}{\samples}$ for samples $\samples$ drawn from the cube skeleton with different amounts of uniform noise (cf. Table~\ref{TAB:signal-to-noise}) as we increase the resolution $m$. We note that in almost all cases, these ratios stabilize at or before $m=50$.

\begin{figure}[h]
    \centering
    \begin{subfigure}{0.44\textwidth}
    \includegraphics[width=\textwidth, trim= 0.5cm 0cm 1cm 0cm, clip]{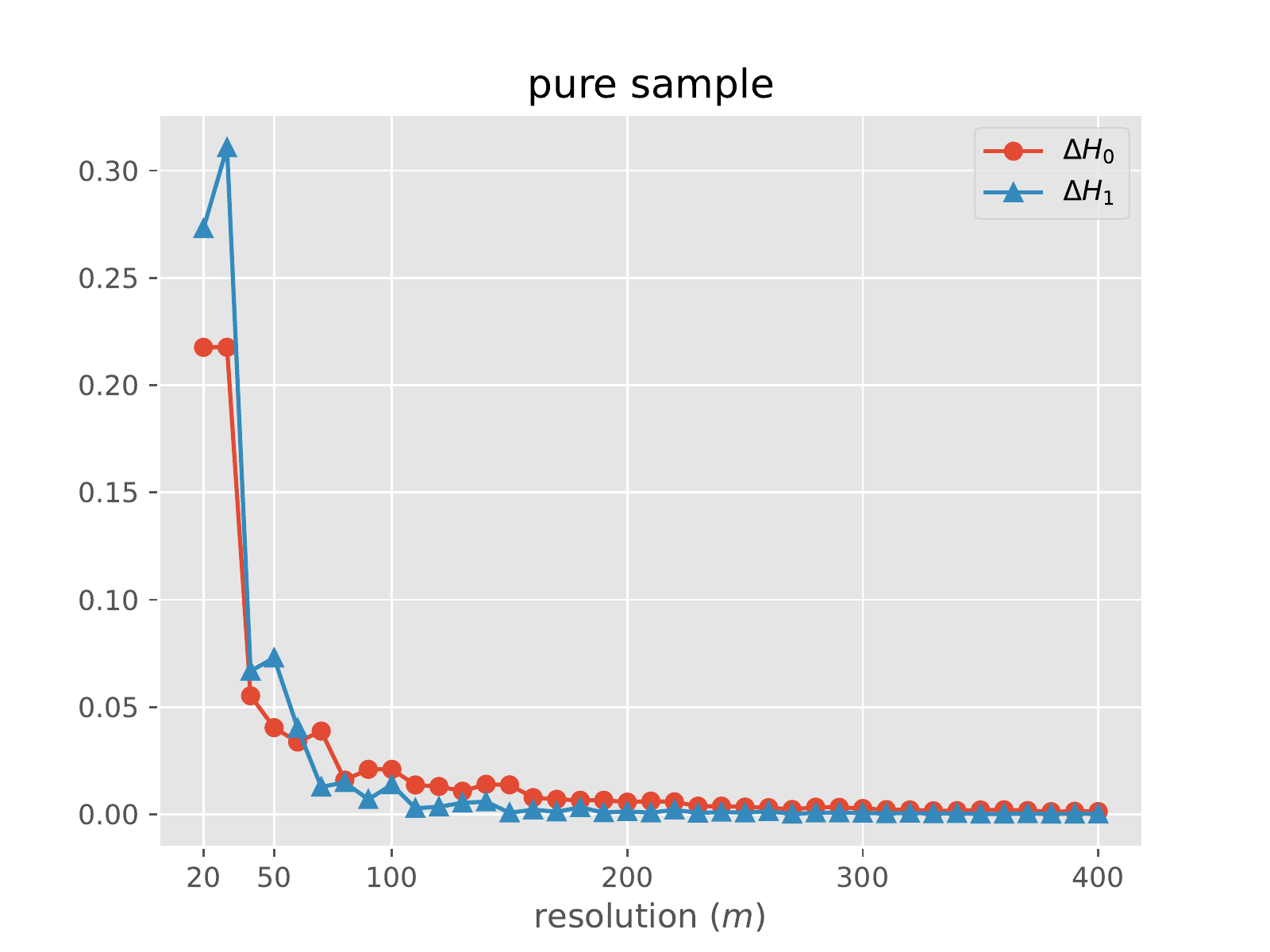}
    \end{subfigure}
    \begin{subfigure}{0.44\textwidth}
    \includegraphics[width=\textwidth, trim=0.5cm 0cm 1cm 0cm, clip]{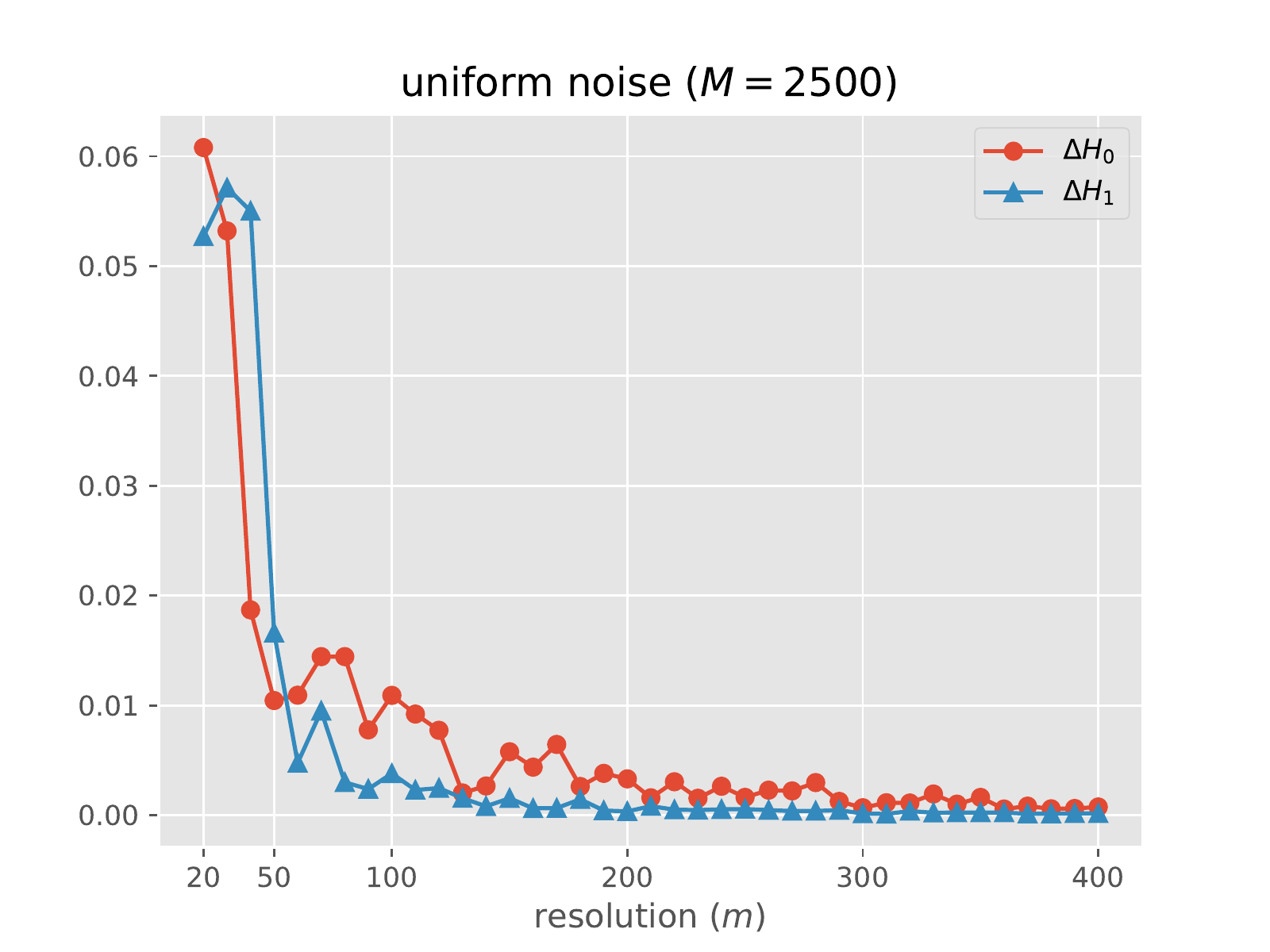}
    \end{subfigure}    \caption{Change in bottleneck distance for $\persmod{12}{\samples}$ as the resolution of the approximation is increased ($\Delta m = 10$). Here, $\samples \subseteq [-1, 1]^2$ is as in Figure~\ref{FIG:fourballs}.}
    \label{FIG:ThreeShapesResIncrease}
\end{figure}

\begin{figure}
    \centering
    \begin{subfigure}{0.44\textwidth}
    \centering
    \includegraphics[width=\textwidth, trim= 0.5cm 0cm 1cm 0cm, clip]{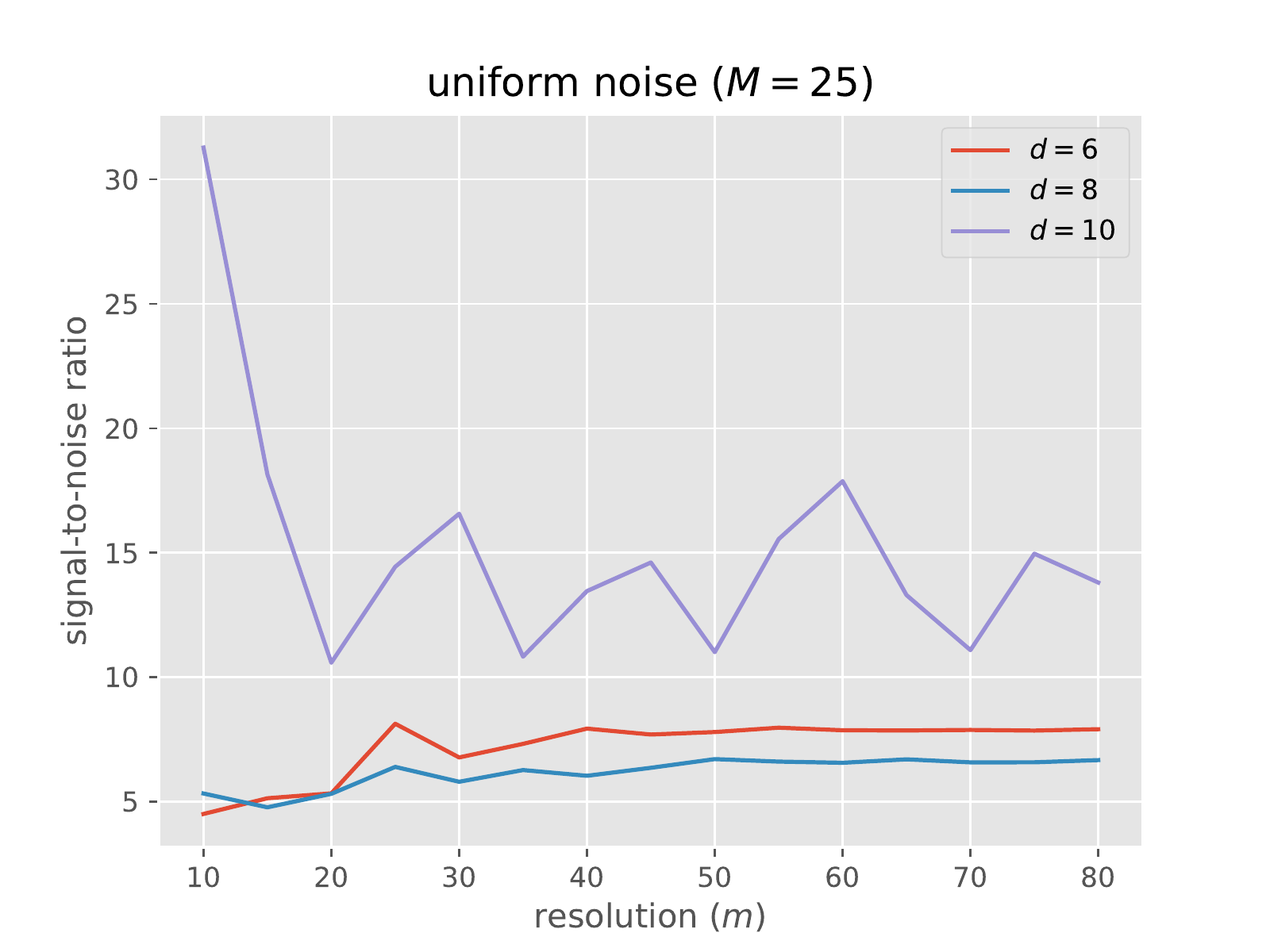}
    \end{subfigure}
    \begin{subfigure}{0.44\textwidth}
    \centering
    \includegraphics[width=\textwidth, trim=0.5cm 0cm 1cm 0cm, clip]{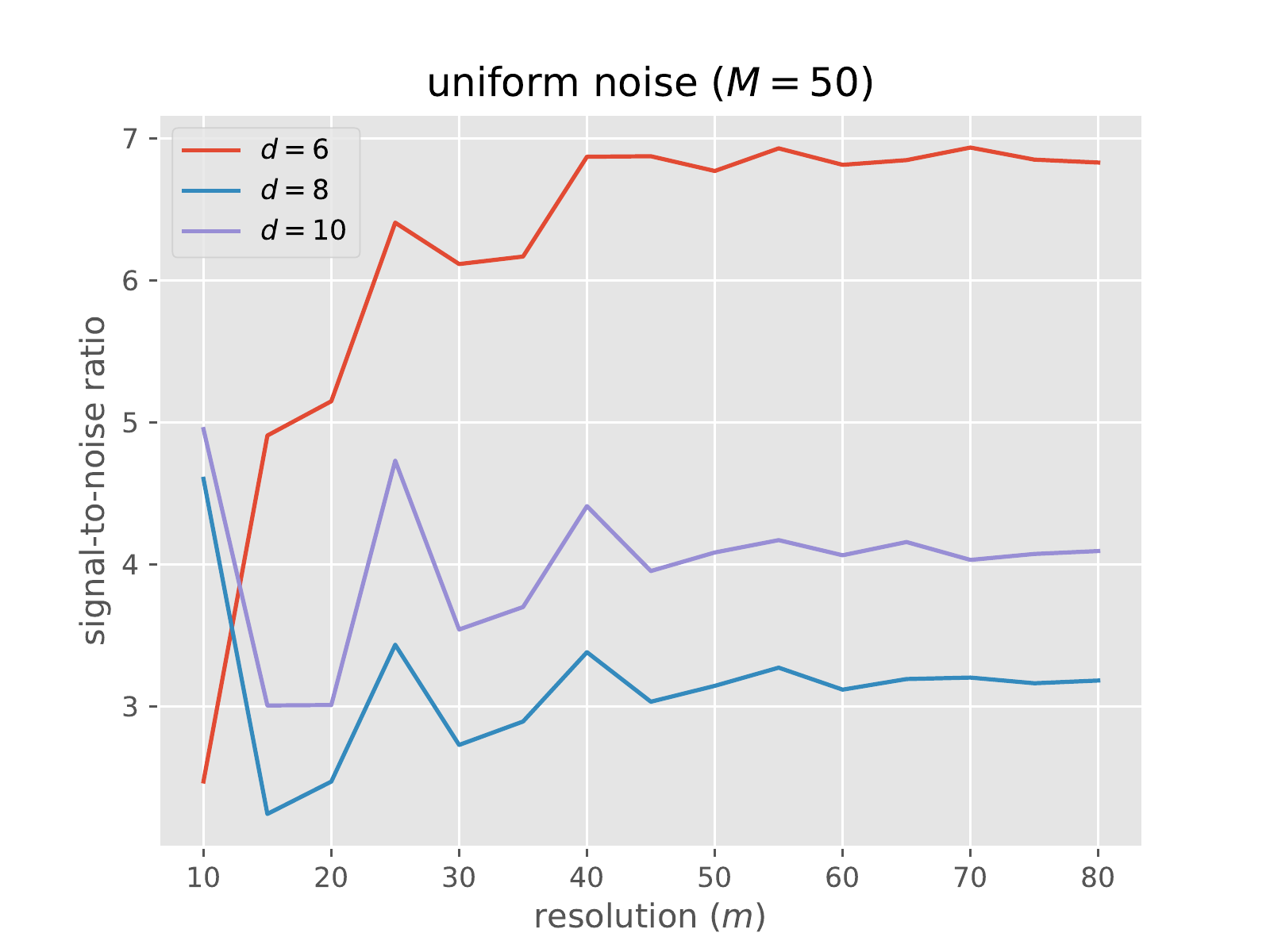}
    \end{subfigure}
    \begin{subfigure}{0.44\textwidth}
    \centering
    \includegraphics[width=\textwidth, trim= 0.5cm 0cm 1cm 0cm, clip]{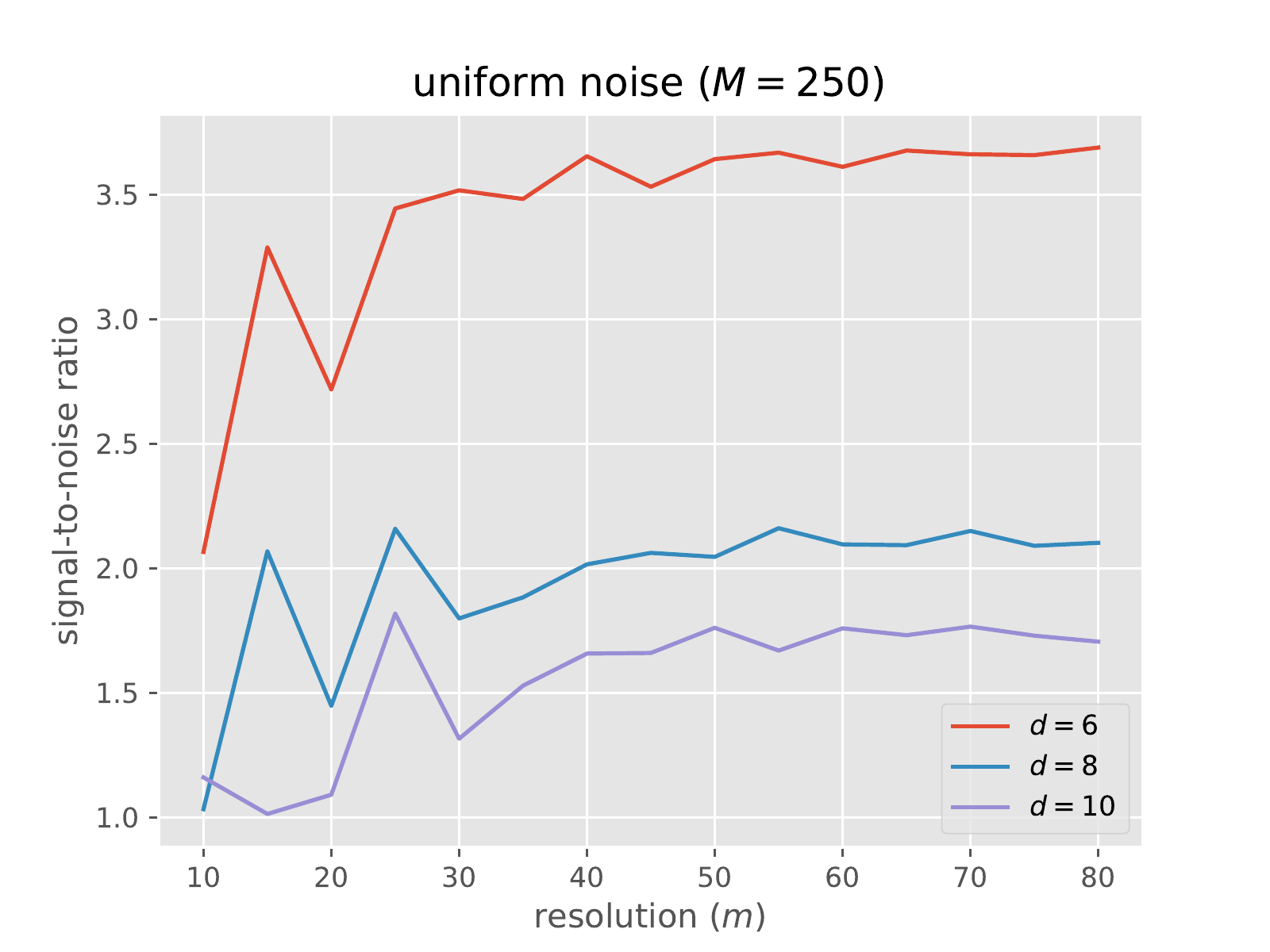}
    \end{subfigure}
    \begin{subfigure}{0.44\textwidth}
    \centering
    \includegraphics[width=\textwidth, trim= 0.5cm 0cm 1cm 0cm, clip]{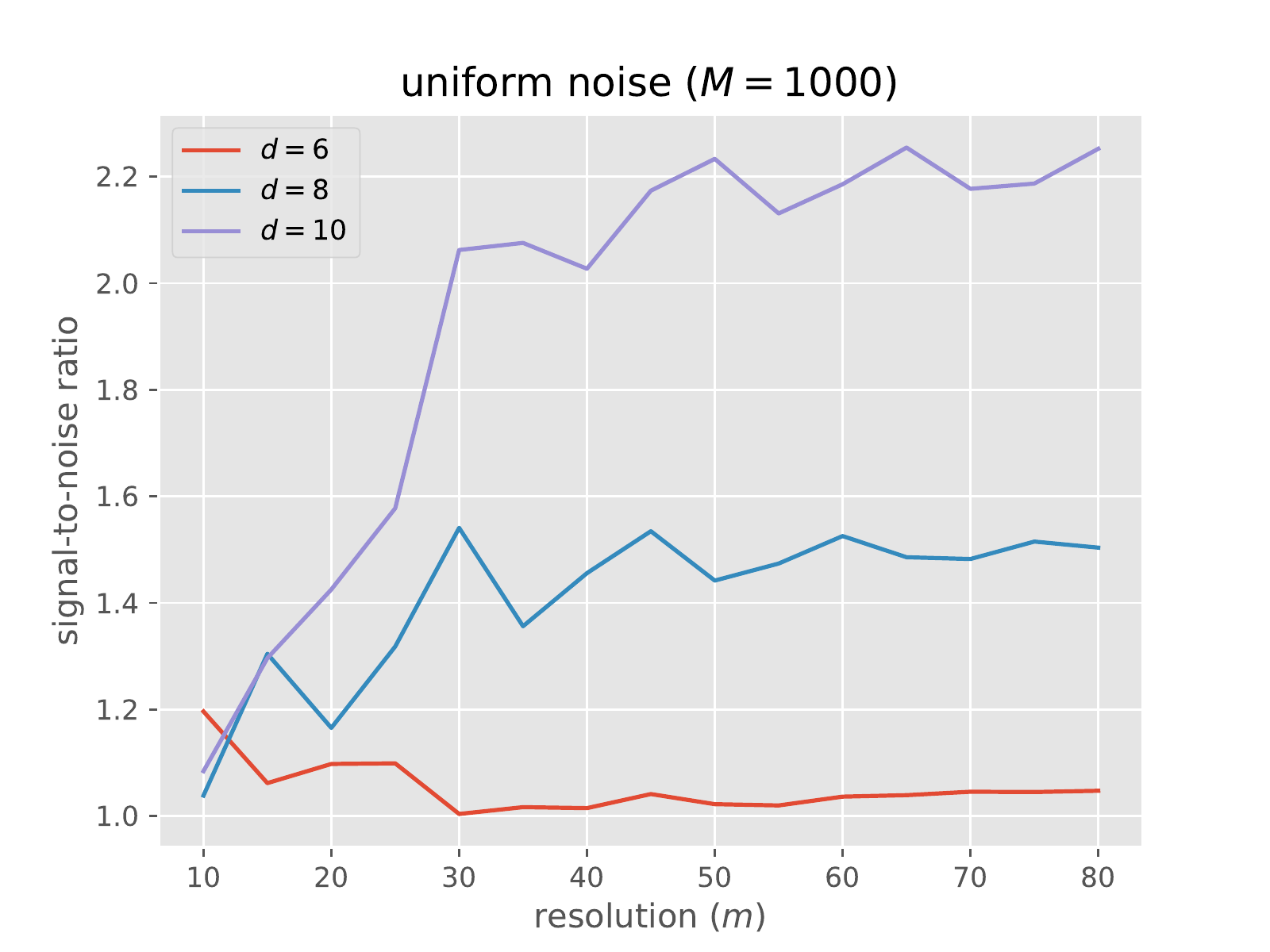}
    \end{subfigure}
    \caption{Observed signal-to-noise ratios in $H_1$ for $\persmod{d}{\samples}$ for the cube skeleton in $[-1, 1]^3$ with uniform noise as the resolution $m$ of the approximation scheme is increased (single experiment).}
    \label{FIG:CubeResIncrease}
\end{figure}

\end{document}